\documentclass[reqno, a4paper]{amsart}
\usepackage[utf8]{inputenc}
\usepackage{hyperref,xy,doi}
\usepackage{amsmath}
\usepackage{amsthm}
\usepackage{amssymb}
\usepackage{enumerate}
\usepackage{xcolor}
\usepackage[parfill]{parskip}

\xyoption{all}
\usepackage{ytableau}

\usepackage{tikz}

\hypersetup{
	colorlinks   = true, 
	urlcolor     = blue, 
	linkcolor    = purple, 
	citecolor   = blue 
}

\newtheorem{theorem}{Theorem}[section]
\newtheorem{corollary}[theorem]{Corollary}
\newtheorem{lemma}[theorem]{Lemma}
\theoremstyle{definition}
\newtheorem{definition}[theorem]{Definition}
\newtheorem{proposition}[theorem]{Proposition}
\theoremstyle{remark}
\newtheorem{remark}[theorem]{Remark}
\newtheorem{example}[theorem]{Example}
\DeclareMathOperator{\Res}{Res}
\DeclareMathOperator{\Ind}{Ind}
\DeclareMathOperator{\res}{res}
\DeclareMathOperator{\ind}{ind}

\DeclareMathOperator{\End}{End}

\DeclareMathOperator{\Mod}{-mod}
\DeclareMathOperator{\sgn}{sgn}
\DeclareMathOperator{\ct}{ct}
\DeclareMathOperator{\id}{id}
\newcommand{\CC}{\mathbb C}
\newcommand{\FF}{\Bbbk}
\newcommand{\NN}{\mathbb N}
\newcommand{\LL}{\mathbb L}
\newcommand{\ZZ}{\mathbb{Z}}
\newcommand{\ty}{\mathrm{wt}}
\newcommand{\inv}{\mathrm{Inv}}
\newcommand{\pa}{\mathrm{Pair}}
\newcommand{\rn}{\mathrm{rank}}

\title[Jucys--Murphy elements for generalized rook monoids]
{Jucys--Murphy elements and Grothendieck groups for generalized rook monoids}


\author[Mazorchuk]{Volodymyr Mazorchuk}
\address{Department of Mathematics, Uppsala University, Box 480, SE75106, Sweden}
\email{mazor@math.uu.se}

\author[Srivastava]{Shraddha Srivastava}
\address{Department of Mathematics, Uppsala University, Box 480, SE75106, Sweden}
\email{maths.shraddha@gmail.com}

\begin{document}

\begin{abstract}
We consider a tower of generalized rook monoid algebras over the field 
$\CC$ of complex numbers  and observe that the Bratteli diagram associated to this 
tower  is a simple graph. We construct simple modules and describe Jucys--Murphy elements 
for  generalized rook monoid algebras.  

Over an algebraically closed field  $\FF$ of positive characteristic  $p$, utilizing Jucys--Murphy elements 
of rook monoid algebras, for $0\leq i\leq p-1$ we define the corresponding $i$-restriction 
and $i$-induction functors along with two extra functors. On the direct sum $\mathcal{G}_{\CC}$ 
of the Grothendieck groups of module categories over rook monoid algebras over $\FF$, 
these functors induce an action of the tensor product of the universal enveloping 
algebra $U(\widehat{\mathfrak{sl}}_p(\CC))$ and the monoid algebra $\CC[\mathcal{B}]$ of the 
bicyclic monoid $\mathcal{B}$. Furthermore, we prove that $\mathcal{G}_{\CC}$ is isomorphic 
to the tensor product of the basic representation of $U(\widehat{\mathfrak{sl}}_{p}(\CC))$ and the unique infinite-dimensional simple  module over $\CC[\mathcal{B}]$, and also exhibit that $\mathcal{G}_{\CC}$ is a bialgebra. Under some natural restrictions on the characteristic of $\FF$, we 
outline the corresponding result for  generalized rook monoids.
\end{abstract}

\maketitle
	
\section{Introduction}\label{s1}

The aim of this paper is to prove several results on the representation theory of certain inverse 
semigroups called generalized rook monoids and on the structure of their semigroup algebras. These results 
are motivated by the corresponding results for the wreath products of symmetric and cyclic groups.
The latter groups appear as maximal subgroups in generalized rook monoids. Below we explain our motivation and
result in more detail. 

Let $R_n$ be the set consisting of all $n\times n$ matrices with entries from $\{0,1\}$ and 
with the further condition that each row and each column contains at most one non-zero entry.  
The matrix multiplication defines on $R_n$ the structure of a monoid, called the
{\em rook monoid}, cf. \cite{Solomon}. The monoid $R_n$ is alternatively known as the 
symmetric inverse semigroup, see \cite{Li,GM}. It is very well-known, see for example 
\cite{Munn}, that the rook monoid algebra $\CC[R_n]$ is semisimple, moreover, all simple modules over
this algebra are very well-understood, see \cite{GM,Steinberg,Grood}.

For a positive integer $r$, let $C_r$ denote the multiplicative cyclic group of order $r$. 
We can consider the wreath product $C_r\wr R_n$, called the {\em generalized rook monoid }
in \cite{Benadvances}, whose elements are all $n\times n$ matrices with entries from $C_r\cup\{0\}$
and with the condition that each row and each column contains at most one non-zero entry. 
Many of the results on the representations of the rook monoid obtained in  \cite{Solomon}
were extended to the case of the generalized rook monoid in \cite{Benadvances}.
	 
Motivated by the construction of the irreducible representations as seminormal representations 
in the case of symmetric groups and generalized symmetric groups, in this article, we give 
a similar construction of the irreducible representations of $C_r\wr R_n$ in Theorem~\ref{thm:constirre}. 
The set of elements of $C_r\wr R_n$ whose $(n,n)$-th entry is equal to $1$ is a submonoid of 
$C_r\wr R_n$ and this submonoid is isomorphic to $C_r\wr R_{n-1}$. Now, viewing $C_r\wr R_{n-1}$ 
as a submonoid of $C_r\wr R_n$ in this way, we have the following tower of generalized 
rook monoid algebras:
\begin{align}\label{al:tower1}
\CC[C_{r}\wr R_0] \subset \CC[C_r\wr R_{1}]\subset  \cdots \subset \CC[C_r\wr R_{n}]\subset \cdots
\end{align}

For $r=1$, the branching rule for the restriction of an irreducible representation for each 
successive inclusion of algebras in \eqref{al:tower1} is multiplicity-free by \cite[Section 3]{Hal04}. 
This means that, in this case, the Bratteli diagram of \eqref{al:tower1} is a simple graph. 
In this article, we prove a similar result for an arbitrary positive integer $r$ in 
Corollary~\ref{coro:branching}. In particular, this gives a natural basis of each
irreducible representation of an algebra in the tower \eqref{al:tower1} indexed by certain paths 
in the Bratteli diagram, usually called the {\em Gelfand--Zeitlin basis}. If we replace, 
in \eqref{al:tower1},  $\CC$ by an algebraically closed  field $\FF$ of positive characteristic, then our method gives 
a modular branching rule as well. We construct a Gelfand model for $\CC[C_r\wr R_n]$ in 
Proposition~\ref{prop:gel} which is a generalization of the case $r=1$ as considered in \cite{MK}, 
see also \cite{MV} and \cite{HR}.

The construction of seminormal representations of the symmetric group $S_n$ is closely connected to the existence of some special elements, called  Jucys--Murphy elements, in the group algebra $\CC[S_n]$, see the introduction of \cite{Ram}.  Jucys--Murphy elements for $\CC[R_n]$ were constructed in \cite{MS21}.
In Section~\ref{sec:JM}, we construct Jucys--Murphy elements for $\CC[C_r\wr R_n]$
(these elements are defined over any field in which $r$ is non-zero).  Moreover, we observe that the expression 
for  Jucys--Murphy elements of $\CC[R_n]$, given in Section~\ref{sec:JM}, is simpler 
than the one in \cite{MS21}. We also show that Jucys--Murphy elements satisfy the fundamental 
properties similar to the ones from the classical setup of symmetric groups. In particular, we have:
\begin{enumerate}[(a)]
\item  Proposition~\ref{prop:JMcommute} shows that these elements commute with each other. 
\item Theorem~\ref{thm:JM} proves that these elements act as scalars on 
all elements of the Gelfand--Zeitlin basis of every simple $\CC[C_r\wr R_n]$-module. 
\item Corollary \ref{coro:irre} states that the eigenvalues of the action of Jucys--Murphy elements on elements of
the Gelfand--Zeitlin basis distinguish non-isomorphic simple modules.
\end{enumerate}

Let now $\FF$ be an algebraically closed field of positive characteristic $p$.  For a finite-dimensional associative 
$\FF$-algebra $A$, let $A\Mod$ denote the category of finite-dimensional left $A$-modules.
Consider the Grothendieck group $K_0(A\Mod)$ of $A\Mod$ and the complexified Grothendieck group $G_0(A)=\CC\otimes_{\ZZ}K_0(A\Mod)$, where $\ZZ$ denotes the ring of integers.

Let $\NN$ denote the set of all non-negative integers. A classical result, proved in  \cite{LLT}, asserts that 
\begin{displaymath}
{\bigoplus_{n\in \NN }}G_0(\FF[S_n])
\end{displaymath}
has the natural structure of a module over the universal 
enveloping algebra $U(\widehat{\mathfrak{sl}}_{p}(\CC))$ of the affine Lie algebra 
$\widehat{\mathfrak{sl}}_{p}(\CC)$ of type $A_{p-1}^{(1)}$. Moreover, this module can be identified as
the basic representation $V(\Lambda_{0})$ of $U(\widehat{\mathfrak{sl}}_{p}(\CC))$. This result was also 
established in  \cite{Groj} for a more general setting with different techniques.  
One of the ways to obtain these results is to define the $i$-restriction and $i$-induction 
functors, for $0\leq i\leq p-1$, using Jucys--Murphy elements of $\FF[S_n]$.  
Then one can show that, at the level of Grothendieck group, the functors satisfy the 
relations for the Chevalley generators of $\widehat{\mathfrak{sl}}_{p}(\CC)$. 
	
Motivated by these classical results,  we use our Jucys--Murphy elements for rook monoid 
algebras to define, for  $0\leq i\leq p-1$, the $i$-restriction functor $\res_{i}$ 
and the $i$-induction functor $\ind_{i}$ in the rook monoid setup, see \eqref{al:E} and \eqref{al:A}.
We also define two extra functors $\mathbb{A}$ and $\mathbb{B}$ which correspond to 
the additional edges in the Bratteli diagram for rook monoids, see \eqref{al:A}.
In Theorem~\ref{thm:main}, we show that, at the level of the direct sum 
\begin{equation}\label{eq1new}
{\bigoplus_{n\in\NN}}G_0(\FF[R_n]),
\end{equation}
of the Grothendieck groups, the functors $\res_{i}$ and $\ind_{i}$, for $0\leq i\leq p-1$, 
satisfy the relations for the Chevalley generators of $\widehat{\mathfrak{sl}}_{p}(\CC)$. 
Additionally, the functors $\mathbb{A}$ and $\mathbb{B}$ satisfy the relation of the 
generators of the bicyclic monoid $\mathcal{B}$ and commute with all 
$\res_{i}$ and $\ind_{i}$. Furthermore, we show that the 
Grothendieck group \eqref{eq1new} is isomorphic to the tensor product of the basic  representation 
$V(\Lambda_0)$ with the unique simple infinite-dimensional $\CC[\mathcal{B}]$-module $V_{\NN}$, 
as modules over $U(\widehat{\mathfrak{sl}}_{p}(\CC))\otimes_{\CC} \CC[\mathcal{B}]$. 

Assume that $r$ is non-zero in $\Bbbk$. Then, using the result for the generalized symmetric group 
algebras similar to the ones proved in \cite{ShuJ} and also in \cite{Wang}, in Subsection~\ref{sec:gen} 
we conclude that the Grothendieck group 
$$\bigoplus_{n\in \NN}G_0(\FF[C_{r}\wr R_n])$$ 
is isomorphic to $V(\Lambda_{0})^{\otimes r}\otimes_{\CC} V_{\NN}$ as a module over
the algebra $U(\widehat{\mathfrak{sl}}_{p}(\CC))^{\otimes r}\otimes_{\CC} \CC[\mathcal{B}]$. 
	
It is well known that $\displaystyle{\bigoplus_{n\in \NN }}G_0(\FF[S_n])$ is a Hopf algebra,
where the multiplication and the comultiplication are obtained by using appropriate induction 
and restriction functors, respectively, e.g. see \cite[Chapter I]{Mac}.  In Theorem \ref{thm:bi}, we prove 
that \eqref{eq1new}  is a bialgebra where the multiplication and the comultiplication are again 
obtained by using certain induction and restriction functors, respectively.
\vspace{3mm}

{\bf Acknowledgments.}
The authors thank Weiqiang Wang for bringing the reference \cite{Wang} to our attention.  
The second author also thanks Arun Ram for fruitful discussions. The first author is partially supported by the Swedish Research Council and 
G{\"o}ran Gustafsson Stiftelse.

\section{Generalized rook monoids}\label{s2}

In what follows, $\Bbbk$ is an algebraically closed field.

Recall that $C_r\wr R_n$ denote the generalized rook monoid. 
For $0\leq i\leq n$, let $f_i\in C_r\wr R_n$ be the diagonal matrix whose $(k,k)$-th 
entry is $0$, when $i+1\leq k\leq n$, and the remaining diagonal entries are equal to $1$. 
Note that $f_0$ and $f_n$ are the zero matrix and the identity matrix in $C_r\wr R_n$, respectively. 

Green's left cell $\mathbb{L}_{i}^{n}$ of $C_r\wr R_n$ corresponding to the idempotent $f_i$
consists, by definition,  of all $\sigma\in C_r\wr R_n$ satisfying
$$(C_r\wr R_n)\sigma=(C_r\wr R_n)f_i.$$
Then $\mathbb{L}_{i}^{n}$ consists of all rank $i$ matrices in $C_r\wr R_n$ whose 
$j$-th column is zero, for all $i+1\leq j\leq n$.  The maximal subgroup of $C_r\wr R_n$ 
corresponding to $f_i$ is the subgroup consisting of matrices in $C_r\wr R_n$, whose non-zero entries lie on the first $i\times i$-block. This subgroup is evidently isomorphic to the generalized symmetric group $C_r\wr S_i$, where $S_i$ is the symmetric 
group on $i$ letters. Unless stated otherwise, we use this identification throughout the manuscript.
Note that $C_r\wr S_i$ acts on $\LL_{i}^n$ from the right in the obvious way.  

Let $[n]:=\{1,2,\ldots,n\}$ and 
$\mathcal{S}_i:=\{Z\subseteq [n]\mid |Z|=i \}$. 
For $Z\in \mathcal{S}_i$, write
\begin{displaymath}
 Z=\{r_1<r_2< \cdots<r_i\},
 \end{displaymath}
 and let 
$h_{Z}^n \in C_r\wr R_n$ be  such that the non-zero entries of $h_Z^n$ are equal 
to $1$ and they are at the coordinates $(r_1,1),\ldots,(r_i,i)$. Note that 
$h_Z^n\in\mathbb{L}_i^n$ and, moreover, these matrices form a cross-section of 
the orbits of the right action of $C_r\wr S_i$ on $\LL_i^n$. 
In other words, $\Bbbk\mathbb{L}_i^n$ is a free right $\FF[C_{r}\wr S_i]$-module which 
a $\FF[C_{r}\wr S_i]$-basis consisting of all matrices of the form $h_Z^n$, where 
$Z\in \mathcal{S}_i$  (we use this basis often in what follows).

The space $\Bbbk\mathbb{L}_i^n$ is also naturally a left $\Bbbk[C_{r}\wr R_n]$-module
where, for $\tau\in C_r\wr R_n$ and $\sigma\in\mathbb{L}_i^n$, the action is given by
\begin{displaymath}
\tau \sigma=
\begin{cases}
\tau\sigma, & \text{ if } \tau\sigma\in\mathbb{L}_{i}^{n};\\
0,  & \text{ otherwise. }
\end{cases}
\end{displaymath}
These two actions on $\Bbbk\mathbb{L}_i^n$ obviously commute, making $\Bbbk\mathbb{L}_i^n$ 
a $(\FF[C_{r}\wr R_n],\FF[C_{r}\wr S_i])$-bimodule. The associated functor
\begin{displaymath}
\mathcal{L}_{i}^n:=(\FF\mathbb{L}_i^n\otimes_{\FF[C_{r}\wr S_i]}{}_-):
\FF[C_{r}\wr S_i]\Mod\to \FF[C_{r}\wr R_n]\Mod
\end{displaymath}
is full, faithful and exact, see \cite[Chapter~4]{Steinberg}.
	
\begin{lemma}\label{lm:Requi}
The following functor is an equivalence of categories
\begin{displaymath}
\bigoplus_{i=0}^{n} \mathcal{L}_i^n:\bigoplus_{i=0}^{n} \FF[C_{r}\wr S_i]\Mod\rightarrow 
\FF[C_{r}\wr R_n]\Mod.
\end{displaymath}
\end{lemma}
	
\begin{proof}
This follows by combining the standard facts that, for $0\leq i \leq n$, 
the right $\FF[C_{r}\wr S_i]$-module
$\FF\mathbb{L}_{i}^{n}$ is  free and that 
$\displaystyle{\bigoplus_{i=0}^{n}}\End_{\FF[C_r\wr S_i]}(\FF\mathbb{L}_i^n)\cong \FF[C_{r}\wr R_n]$,
see \cite[Section~10.2]{Steinberg}.
\end{proof}

{\emph{Generators}.} \label{gn}For $1\leq j\leq n-1$, let $s_{j}$ denote the simple 
transposition $(j,j+1)$ in $S_n$. Fix a primitive $r$-th root of unity $\xi$ in $C_r$. 
Denote by $P\in C_r\wr S_n$ the diagonal matrix whose $(1,1)$-th entry is $0$ and the 
remaining diagonal entries are equal to $1$. Denote by $Q\in C_r\wr S_n$ the diagonal matrix 
whose $(1,1)$-th entry is $\xi$ and the remaining diagonal entries are equal to $1$. 
Then it is easy to check that $C_r\wr S_n$ is generated by $P$, $Q$ and all $s_{j}$, 
where $1\leq j\leq n-1$.  
	
\section{Seminormal representations}\label{s3}

\subsection{Bases of irreducible representations}\label{s3.1}

In this section, we construct the irreducible representations of $C_r\wr R_n$ over $\CC$, 
give a basis of an irreducible representation of $C_r\wr R_n$ and describe the actions 
of generators of $C_r\wr R_n$. We also give the branching rule 
for the restriction of an irreducible representation of $C_r\wr R_n$ to $C_r\wr R_{n-1}$. 
To obtain these results we need the following notation and definitions. 

Let $\mathcal{P}$ denote the set of all partitions of all non-negative integers. Given a partition $\lambda=(\lambda_{1},\ldots,\lambda_{r})$ of a positive integer, its Young diagram $[\lambda]$ is given as:
\begin{displaymath}
\{(p,q)\mid 1\leq p\leq r \text{ and } 1\leq q\leq \lambda_{p}\}.
\end{displaymath}
We use the usual English notation for Young diagrams.
The elements of $[\lambda]$ are called {\em boxes}. 
By convention, the Young diagram of $0$ is denoted $\varnothing$. 
For $\lambda\in \mathcal{P}$, let $|\lambda|$ denote the number of boxes 
in $[\lambda]$. For $\lambda\in \mathcal{P}$ with $|\lambda|\leq n$, let $\mathcal{Y}(\lambda,n)$ 
denote the set of all fillings of boxes of $[\lambda]$ with different elements from 
$[n]$ such that the entries  increase along the rows from left to right and along 
the columns from top to bottom.  Let 
\begin{displaymath}
\Lambda_r(n):=\bigg\{\lambda^{(r)}=(\lambda_{(1)},\ldots,\lambda_{(r)})\mid 
\lambda_{(i)}\in\mathcal{P}, \, \text{ for } 1\leq i\leq r , \text{ and } \sum_{i=1}^{r}|\lambda_{(i)}|=n \bigg\}.
\end{displaymath}

Let $\lambda^{(r)}\in\Lambda_r(n)$ and $m$ be a non-negative integer such that $n\leq m$. Define
$\mathcal{Y}(\lambda^{(r)},m)$ as the set
\begin{displaymath}
\left\{(L_1,\ldots,L_r)\,\bigg|
\begin{array}{l}
L_i\in\mathcal{Y}({\lambda_{(i)}},m), \text{ for } 1\leq i\leq r;\\
\text{$L_i$ and $L_j$ don't have common entries for }  1\leq i\neq j\leq r
\end{array}
\right\}.
\end{displaymath}
Let $L=(L_1,\ldots,L_r)\in\mathcal{Y}(\lambda^{(r)},m)$ and $1\leq b\leq m$. 
We write $b\in L$ if $b$ appears in one of  $L_k$, for $1\leq k\leq r$, and we 
also say ``$b\in L$ at the position $k$", if $b$ appears in $L_k$.
Define the sign of $b$ in $L$ as
\begin{displaymath}
\sgn_{L}(b):=\begin{cases}
\xi^{k-1}, & \text{ if } b\in L \text{ at the position }k;\\
0, & \text{ otherwise.}
\end{cases}
\end{displaymath}
Let $b\in L$ be at the position $k$ and, further, assume that $b$ is in the box 
$(s_1,s_2)$ in $L_k$. Define the content of  $b$ as $\ct(L(b)):=(s_2-s_1)$. If both $i$ 
and $i+1$ appear in $L$ at the position $k$ (in particular, $\sgn_L(i)=\sgn_L(i+1)$), 
define
\begin{displaymath}
a_{L}(i):=\frac{1}{\ct(L(i+1))-\ct(L(i))}.
\end{displaymath}
Given $L\in\mathcal{Y}(\lambda^{(r)},n)$, let $s_iL$ be obtained from $L$ by replacing $i$ by $i+1$ if $i\in L$, and by replacing $i+1$ by $i$ if $i+1\in L$. Note that it may happen that $s_iL$ does not lie in $\mathcal{Y}(\lambda^{(r)},n)$. For the next statement, we refer e.g. to \cite[Page~169]{HR98}, see also \cite{AK}.

\begin{theorem}\label{thm:gensym}
{\hspace{1mm}}

\begin{enumerate}[$($a$)$]
\item\label{thm:gensym.1} The elements of $\Lambda_r(n)$ index the isomorphism classes 
of irreducible representations of $\mathbb{C}[C_r\wr S_n]$. 
\item\label{thm:gensym.2} For $\lambda^{(r)}\in\Lambda_{r}(n)$, the 
corresponding irreducible representation $W_{\lambda^{(r)}}^n$ of $C_r\wr S_n$ has a 
basis $\{w_{L}\vert L\in\mathcal{Y}(\lambda^{(r)},n)\}$ on which 
the  generators $s_{j}$, for $1\leq j\leq n-1$, and $Q$ act as follows:
\begin{align}
&s_{j}  w_{L}=\begin{cases}
w_{s_jL}, & \text{ if } \sgn_{L}(j)\neq\sgn_{L}(j+1);\\
a_{L}(j)w_{L}+(1+a_{L}(j))w_{s_jL},  &\text{ if } \sgn_{L}(j)=\sgn_{L}(j+1);
\end{cases}\label{al:symsi}\\
& Q  w_{L}=\begin{cases}
\xi^{k-1}w_{L}, & \text{ if } 1\in L \text{ at the position }k;\\
w_{L}, & \text{ otherwise. } 
\end{cases}\label{al:symQ}
\end{align}
Here $w_{s_jL}=0$, if $s_jL\notin\mathcal{Y}({\lambda^{(r)},n})$.
\end{enumerate}
\end{theorem}
The next claim is a generalization of Theorem~\ref{thm:gensym} to the case of $C_r\wr R_n$.

\begin{theorem}\label{thm:constirre}
{\hspace{1mm}}

\begin{enumerate}[$($a$)$]
\item\label{thm:constirre.1}
The elements of $\displaystyle \Lambda_{r}(\leq\hspace{-1mm} n):=\bigcup_{i=0}^{n}\Lambda_{r}(i)$ index the 
irreducible representations of $\mathbb{C}[C_r\wr R_n]$ in the following way:
for $\lambda^{(r)}\in\Lambda_r(i)$, the corresponding irreducible representation is 
$$V_{\lambda^{(r)}}^n:=\mathbb{C}\mathbb{L}_i^n\otimes_{\CC[C_r\wr S_i]} W_{\lambda^{(r)}}^i.$$
\item\label{thm:constirre.3}
$V_{\lambda^{(r)}}^n$ has a basis 
$\{v_{L}\vert L\in \mathcal{Y}(\lambda^{(r)},n)\}$ on which the generators $P$, $Q$, and $s_j$, 
for $1\leq j\leq n-1$, act as follows:
\begin{align}
&s_{j}  v_{L}=\begin{cases}
v_{s_jL}, & \text{ if } j\in L, j+1\notin L;\\
v_{s_jL}, & \text{ if } j\notin L, j+1\in L;\\
v_{L}, & \text{ if } j\notin L, j+1\notin L;\\
v_{s_jL}, & \text{ if } \begin{array}{ll}j\in L, j+1\in L\text{ and }\\ \sgn_{L}(j)\neq\sgn_{L}(j+1);\end{array}\\
a_{L}(j)v_{L}+(1+a_{L}(j))v_{s_jL},  &\text{ if }\begin{array}{ll}j\in L, j+1\in L\text{ and }\\ 
\sgn_{L}(j)=\sgn_{L}(j+1);\end{array}
\end{cases}\label{al:rooksi}\\
&P  v_{L}=\begin{cases}\label{al:P-Q}
v_{L}, & \text{ if } 1\notin L;\\
0, & \text{ otherwise; } 
\end{cases}, \qquad
Q  v_{L}=\begin{cases}
\xi^{k-1}v_{L}, & \text{ if } \begin{array}{ll}1\in L \text{ at the}\\ \text{position }k;\end{array}\\
v_{L}, & \text{ otherwise. } 
\end{cases}
\end{align}
\end{enumerate}
Here $v_{s_jL}=0$, if $s_jL\notin\mathcal{Y}({\lambda^{(r)},n})$.
\end{theorem}

\begin{proof}
The first claim follows directly from the general  theory, see \cite{Steinberg},
so we only prove the second claim. (One can also see it by combining Lemma \ref{lm:Requi} and Theorem \ref{thm:constirre}.)
Recall that, as a right $\FF[C_{r}\wr S_i]$-module, $\FF\mathbb{L}_i^n$ has a basis 
consisting of matrices  of the form $h_Z^n$, where $Z\in \mathcal{S}_i$. 

Fix $\lambda^{(r)}\in \Lambda_r(i)$.
For $Z=\{r_1<r_2< \cdots <r_i\}\subseteq [n]$ and $L'\in \mathcal{Y}(\lambda^{(r)},i)$, 
define $L\in\mathcal{Y}(\lambda^{(r)},n)$ by replacing $l\in L$ by $r_l$, 
for all $1\leq l\leq i$. Conversely, given $L\in\mathcal{Y}(\lambda^{(r)},n)$,
let $Z$ be the set of the entries in $L$. We can  arrange these entries in the increasing order 
to get $Z=\{r_1<r_2<\cdots<r_i\}$.  Now, replacing $r_l\in L$ by 
$l$, we obtain an element $L'\in \mathcal{Y}(\lambda^{(r)},i)$. Then 
\begin{displaymath}
\{v_{L}:=h_Z^n\otimes w_{L'}\mid Z\in \mathcal{S}_i \text{ and } 
L'\in\mathcal{Y}(\lambda^{(r)},i) \}=\{v_{L}\mid L\in\mathcal{Y}(\lambda^{(r)},n)\} 
\end{displaymath}
is, by construction, a basis of $V^{n}_{\lambda^{(r)}}$.

Next we  compute the action of $s_{j}$. For $1\leq j\leq n-1$, we have:
\begin{displaymath}
s_{j}  v_{L}=s_{j}(h_{Z}^n\otimes w_{L'})=s_{j}h_{Z}^n\otimes w_{L'}=
h_{s_{j}(Z)}^n\otimes (h_{s_{j}(Z)}^n)^{tr} s_{j}h_{Z}^nw_{L'},
\end{displaymath}
where $(h_{s_{j}(Z)}^n)^{tr}$ denotes the transpose of $h_{s_j(Z)}^n$.

{\bf Case~1.} Suppose that we have $j\notin L$ or $j+1\notin L$. 
This means that $j\notin Z$ or $j+1\notin Z$,  respectively. 
In this case, $(h_{s_{j}(Z)}^n)^{tr}s_{j}h_{Z}^n=f_i\in C_r\wr S_i$, which is the identity of $C_r\wr S_i$. Therefore, we have
$$h_{s_{j}(Z)}^n\otimes (h_{s_{j}(Z)}^n)^{tr}s_{j}h_{Z}^nw_{L'}=h_{s_{j}(Z)}^n\otimes w_{L'}=v_{s_{j}L}.$$
Furthermore, if both $j\notin L$ and $j+1\notin L$, then $v_{s_{j}L}=v_{L}$.
This completes the description of the action of $s_j$ for the first three cases in \eqref{al:rooksi}.

{\bf Case~2.} Suppose that $j\in L$ and $j+1\in L$ or, equivalently, $j\in Z$ and $j+1\in Z$. 
Then $(h_{s_{j}(Z)}^n)^{tr}s_{j}h_{Z}^n$ is a $(j,j+1)$ transposition in $C_r\wr S_i$. Then the remaining two cases in \eqref{al:rooksi}
follow from \eqref{al:symsi}.

To compute the action of $P$, we start with $P  v_{L}=P(h_{Z}^n\otimes w_{L'})=Ph_{Z}^n\otimes w_{L'}$. 
Note that $Ph_{Z}^n\in \mathbb{L}_{i}^n$ if and only if $1\notin Z$. In particular, we have
\begin{align*}
Ph_{Z}^n=\begin{cases}
h_{Z}^n, &\text{ if } 1\notin Z;\\
0, &\text{ otherwise }.
\end{cases}
\end{align*}
This implies the formula for the action of $P$ in \eqref{al:P-Q}.	

The action of $Q$ in \eqref{al:P-Q} can be computed similarly using \eqref{al:symQ}.
\end{proof}

\subsection{The restriction functor}\label{s3.2}

As we already mentioned, the set consisting of all matrices in $C_r\wr R_{n}$ whose 
$(n,n)$-th entry is $1$ is a submonoid of $C_r\wr R_n$ and it is isomorphic to $C_r\wr R_{n-1}$. 
This defines an embedding $C_r\wr R_{n-1}\subset C_{r}\wr R_n$. Similarly, we have $C_r\wr S_{n-1}\subset C_r\wr S_n$.

Denote by $\mathcal{F}$ the functor
\begin{displaymath}
\mathcal{F}:\bigoplus_{i=0}^{n} \FF[C_{r}\wr S_i]\Mod\to 
\bigoplus_{j=0}^{n-1} \FF[C_{r}\wr S_j]\Mod 
\end{displaymath}
given by
\begin{align*}
\mathcal{F}|_{{K[C_r\wr S_i]\Mod}}=
\begin{cases}
\mathrm{Id}_{\FF[C_{r}\wr S_0]\Mod}, & \text{ if } i=0;\\
& \\
\Res^{\FF[C_{r}\wr S_{n}]}_{\FF[C_{r}\wr S_{n-1}]}, & \text{ if } i=n;\\
& \\
\mathrm{Id}_{\FF[C_{r}\wr S_i]\Mod}\oplus\Res^{\FF[C_{r}\wr S_{i}]}_{\FF[C_{r}\wr S_{i-1}]}, & \text{ if } 0<i<n.
\end{cases}
\end{align*}

\begin{theorem}\label{thm:res}
 The following diagram 
\begin{displaymath}
\xymatrixcolsep{9pc}
\xymatrix{	
\displaystyle{\bigoplus_{i=0}^{n}} \FF[C_{r}\wr S_i]\Mod\ar[d]^{\mathcal{F}} 
\ar[r]<1pt>^{\displaystyle{\bigoplus_{i=0}^{n}} \mathcal{L}_i^n}
&\FF[C_{r}\wr R_n]\Mod\ar[d]^{\Res^{\FF[C_{r}\wr R_n]}_{\FF[C_{r}\wr R_{n-1}]}}\\
\displaystyle{\bigoplus_{j=0}^{n-1}} \FF[C_{r}\wr S_j]\Mod 
\ar[r]<1pt>^{\displaystyle{\bigoplus_{j=0}^{n-1}} \mathcal{L}_j^{n-1}}
&\FF[C_{r}\wr R_{n-1}]\Mod
}
\end{displaymath}
commutes up to a natural isomorphism of functors.
\end{theorem}

\begin{proof}
Recall that, as a right $\FF[C_{r}\wr S_i]$-module, $\FF\mathbb{L}_i^n$ has a basis 
 consisting of matrices of the form $h_Z^n$, where $Z\in\mathcal{S}_i$. We need to consider several cases.

{\bf Case~1.} Assume $0<i<n$ and let $V\in \FF[C_{r}\wr S_i]\Mod$. Then the linear span
of all $h_{Z}^n\otimes V$, where $Z\in\mathcal{S}_i$ is such that $n\notin Z$,
is a subspace of $\mathcal{L}_{i}^n(V)$ which is stable under the action of $\CC[C_{r}\wr R_{n-1}]$.
It is easy to see that this $\CC[C_{r}\wr R_{n-1}]$-module is isomorphic to $\mathcal{L}^{n-1}_{i}(V)$. 

Similarly, the linear span
of all $h_{Z}^n\otimes V$, where $Z\in\mathcal{S}_i$ is such that $n\in Z$,
is a subspace of $\mathcal{L}_{i}^n(V)$ which is stable under the action of $\CC[C_{r}\wr R_{n-1}]$.
It is easy to see that this $\CC[C_{r}\wr R_{n-1}]$-module is isomorphic to $\mathcal{L}^{n-1}_{i-1}(V)$. 

{\bf Case~2.} In the case $i=n$, we have that  $\Bbbk\mathbb{L}_{n}^{n}$ is the right regular 
$\CC[C_{r}\wr S_n]$-module and the necessary claim follows from the construction.

{\bf Case~3.} In the case $i=0$, both $\Bbbk\mathbb{L}_{0}^{n}$ and the group algebra of $C_{r}\wr S_0$ 
are isomorphic to $\Bbbk$ and the claim is trivial. 
\end{proof}

Now we give some applications of Theorem \ref{thm:res}.

For $\lambda\in \mathcal{P}$, an {\em outer corner} of $[\lambda]$ 
(a.k.a. {\em removable node}) is a box $(i,j)\in[\lambda]$ such that 
$[\lambda]\setminus\{(i,j)\}$ is still a Young diagram. 
For $\lambda^{(r)}\in\Lambda_{r}(\leq\hspace{-1mm} n)$, we define 
$(\lambda^{(r)})^{-,=}$ as the set consisting of $\lambda^{(r)}$ and all
elements in $\Lambda_{r}(\leq\hspace{-1mm} n-1)$ which are obtained from $\lambda^{(r)}$ by removing an outer corner 
in one of the Young diagrams which constitute $\lambda^{(r)}$. 
Further, let $(\lambda^{r})^{-}:=(\lambda^{r})^{-,=}\setminus\{\lambda^{(r)}\}$.

\begin{corollary}[Branching rule over $\mathbb{C}$]\label{coro:branching} 
For $\lambda^{(r)}\in\Lambda_r(\leq n)$, we have:
\begin{displaymath}
\Res^{\CC[C_r\wr R_n]}_{\CC[C_r\wr R_{n-1}]}(V^{n}_{\lambda^{(r)}})\cong 
\begin{cases}
\displaystyle{\bigoplus_{\mu^{(r)}\in(\lambda^{(r)})^{-,=}}} V^{n-1}_{\mu^{(r)}}, 
& \text{ if } \lambda^{(r)}\in\Lambda_{r}(\leq\hspace{-1mm} n-1);\\
\displaystyle{\bigoplus_{\mu^{(r)}\in(\lambda^{(r)})^{-}}} V^{n-1}_{\mu^{(r)}}, 
& \text{ if }\lambda^{(r)}\in\Lambda_{r}(n).
\end{cases}
\end{displaymath}
\end{corollary}
\begin{proof}
	It is a consequence of the branching rule for $\CC[C_r\wr S_{n-1}]\subset\CC[C_r\wr S_n]$ and Theorem \ref{thm:res}.
	\end{proof}

The Bratteli diagram of the tower \eqref{al:tower1} is an undirected graph whose 
vertices at the level $n$ are given by the elements of 
$\Lambda_r(\leq\hspace{-1mm}n)$.  For two vertices
$\lambda^{(r)}\in\Lambda_r(\leq\hspace{-1mm} n)$ and 
$\mu^{(r)}\in\Lambda_r(\leq\hspace{-1mm} n-1)$, 
there is an edge between $\mu^{(r)}$ and $\lambda^{(r)}$ if and only if 
$\mu^{(r)}\in(\lambda^{(r)}){^{-,=}}$, cf. \cite[Page~584]{OV}. A path from 
the vertex $(\varnothing,\varnothing,\ldots,\varnothing)$, at the level $m=0$,
to the vertex $\lambda^{(r)}$, at the level $m=n$, in the Bratteli diagram 
is a list $(\nu^{(r)}_{0},\nu^{(r)}_1,\ldots,\nu^{(r)}_n=\lambda^{(r)})$ of 
vertices such that the vertex $\nu^{(r)}_i$ is at the level $m=i$ for 
$0\leq i\leq n$, and there is an edge between $\nu^{(r)}_j$ and $\nu^{(r)}_{j+1}$, 
for $0\leq j\leq n-1$. By construction, the Bratteli diagram encodes the branching 
rule in Corollary~\ref{coro:branching}. In order to exhibit the Bratteli diagram, often it is more intuitive to consider the Young diagram corresponding to a partition and in the below we follow this.

In Figure~\ref{fig:BDRM}, we illustrate the branching rule in the case $r=2$
by the corresponding Bratteli diagram for the tower of generalized rook monoid algebras, 
up to the second level.

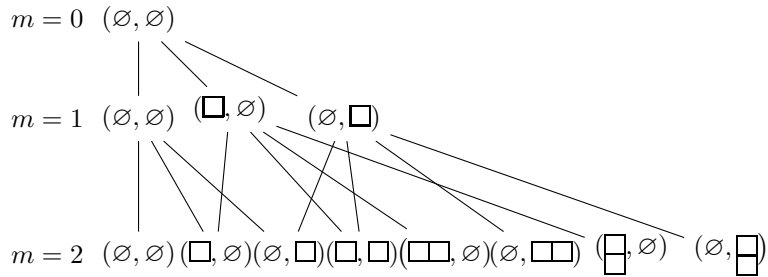
\begin{figure}[ht]
\begin{tikzpicture}[scale=0.6]
    \node (a) at (-2,0) {$m=0$};
    \node (b) at (-2,-2.2) {$m=1$};
    \node (c) at (-2,-5.2) {$m=2$};
    \node (1) at (0,0) {$(\varnothing,\varnothing)$};
    \node (1') at (0,-2.2) {$(\varnothing,\varnothing)$};
    \node (2) at  (2,-2)  {$ (\ytableausetup{boxsize=0.7em}{\begin{ytableau}
    \empty \\
    \end{ytableau}},\varnothing)$};
    \node (2'') at  (4.5,-2.2) {$ (\varnothing, \ytableausetup{boxsize=0.7em}\begin{ytableau}
    \empty \\
    \end{ytableau})$};
    \node (2') at (0,-5.2) {$(\varnothing,\varnothing)$};
    \node (3) at  (1.7,-5.2) {$(\ytableausetup{boxsize=0.7em}\begin{ytableau}
    \empty \\
    \end{ytableau},\varnothing)$};
    \node (3') at  (3.3,-5.2) {$ (\varnothing, \ytableausetup{boxsize=0.7em}\begin{ytableau}
    \empty \\
    \end{ytableau})$};
    
    \node (3'') at (4.9,-5.2) {$ (\ytableausetup{boxsize=0.7em}\begin{ytableau}
    \empty \\
    \end{ytableau},   \ytableausetup{boxsize=0.7em} \begin{ytableau}
	\empty \\
	\end{ytableau})$};
	\node (4) at (6.7,-5.2) {$ \big(\ytableausetup{boxsize=0.7em}\begin{ytableau}
	\empty & \empty \\
	\end{ytableau},\varnothing)$};
	\node (4') at (8.7,-5.2) {$(\varnothing, \ytableausetup{boxsize=0.7em}\begin{ytableau}
	\empty & \empty \\
	\end{ytableau})$};
	\node (5) at (10.8,-5.2) {$ (\ytableausetup{boxsize=0.7em}\begin{ytableau}
	\empty \\
	\empty \\
	\end{ytableau},\varnothing)$};
				
	\node (5') at (13,-5.2) {$(\varnothing, \begin{ytableau}
	\empty \\
	\empty \\
	\end{ytableau}\big)$};
	\draw (1) -- (2'');
	\draw (1) -- (1') -- (2');
	\draw (1) -- (2) -- (3);
	\draw (2) -- (4);
	\draw (2) -- (5);
	\draw (1') -- (3);
	\draw (1') -- (3');
	\draw (2'')--(3'');
	\draw (2) --(3'');
	\draw (2'')--(3');
	\draw (2'')--(4');
	\draw (2'')--(5');
\end{tikzpicture}
\caption{Bratteli diagram for the tower of generalized rook monoid algebras in the case $r=2$, up to level $2$}
\label{fig:BDRM}
\end{figure} 

Observe that the branching rule in Corollary~\ref{coro:branching} is 
multiplicity-free. Therefore there is a basis 
of $V^{n}_{\lambda^{(r)}}$, defined uniquely up to rescaling of its elements,
which is indexed by the paths from the vertex at the 
level $m=0$ to a vertex, say $\lambda^{(r)}$, at the level $m=n$ in the 
Bratteli diagram, see e.g. \cite[Page~585]{OV}, where such a basis is 
called a {\em Gelfand--Zeitlin basis}. 
We note that the set of all such paths is in a 
bijective correspondence with $\mathcal{Y}(\lambda^{(r)},n)$. 
From Theorem~\ref{thm:JM} which will be proved later, it follows that 
the basis constructed in Theorem~\ref{thm:constirre} with a Gelfand--Zeitlin 
basis of $V^n_{\lambda^{(r)}}$. For $L\in\mathcal{Y}(\lambda^{(r)},n)$, 
the vector $v_{L}$ given in the part of Theorem \ref{thm:constirre}  
is called a {\em Gelfand--Zeitlin basis vector}.

Now we outline the modular branching rule for generalized rook monoids as a consequence of Theorem \ref{thm:res} and the corresponding rule for generalized symmetric groups. Let $p$ be a prime number and assume that 
$\Bbbk$ is of characteristic $p$.
Recall that a partition is called {\em $p$-regular} if it does not have more than $p-1$ parts that are equal.   
Let $\Lambda^{p}(n)$ denote the set of all $p$-regular partitions of $n$. It is known that the simple
modules of $\FF[S_n]$ are indexed by the elements of $\Lambda^{p}(n)$, see \cite{James76} or 
\cite[Theorem 11.5]{Jbook}. For $\lambda\in\Lambda^{p}(n)$, let $D^{\lambda}$ denote the 
corresponding simple $\FF[S_n]$-module.  In \cite{K1,K2} one can find a modular branching 
rule for symmetric groups, that is a description of the socle of $\Res^{\FF[S_n]}_{\FF[S_{n-1}]}(D^{\lambda})$.
In particular, this asserts that  the socle of $\Res^{\FF[S_n]}_{\FF[S_{n-1}]}(D^{\lambda})$ 
is multiplicity-free. A classification of simple modules over $\FF[C_r\wr S_n]$ can be found in \cite{JK}. 
A modular branching rule result for generalized symmetric group, under the assumption that $r$ is non-zero in $\FF$,  was obtained in \cite{ShuJ}. 
Combining these results with Theorem~\ref{thm:res}, one obtains a modular branching rule for the 
generalized rook monoid algebras, in particular, it follows that this branching is multiplicity-free.

\subsection {Gelfand model for $\CC[C_r\wr R_n]$}\label{s3.3}
    
A {\em Gelfand model} of a finite-dimensional semisimple algebra $A$ is a 
multiplicity-free additive generator of $A$-mod, see \cite{PR,MV,HR} for further details.
	 
For $\sigma\in C_r\wr S_n$, let $\inv(\sigma)$ be the set consisting of all 
$(i,j)\in[n]\times[n]$ such that $i<j$ and the non-zero entry in 
the $i$-th column of $\sigma$ appears  in a later row than the non-zero entry in the $j$-th column. 
Similarly, for $\sigma\in C_r\wr S_n$, let $\pa(\sigma)$ be the set consisting of all pairs 
$(i,j)\in[n]\times[n]$ such that $i<j$,
the non-zero entry in the $i$-th column of $\sigma$ appears in row $j$ and 
the non-zero entry in the $j$-th column appears in row $i$. For a matrix $A$, let $A^{tr}$ denote its transpose. 
Let $\mathcal{W}_n$ be the $\CC$-span of $\mathcal{J}=\{\sigma\in C_r\wr S_n\mid \sigma=\sigma^{tr}\}$.
For $\omega\in C_r\wr S_n$ and $\sigma\in \mathcal{J}$, set
\begin{align*}
\omega  \sigma=(-1)^{|\inv(\omega)\cap\pa(\sigma)|}\omega \sigma \omega^{-1}.
\end{align*}
In \cite{APR}, it is shown that $\mathcal{W}_n$ is a Gelfand model for $\CC[C_r\wr S_n]$, 
see also \cite[Section 2.7]{VS}.
	 
Let $\mathcal{V}$ be the $\CC$-span of $\mathcal{I}=\{M\in C_r\wr R_n\mid M=M^{tr}\}$.	 
For $M\in C_r\wr R_n$, let $\pa(M)$ denote the set consisting of all pairs 
$(i,j)\in[n]\times[n]$ such that $i<j$,  the non-zero entry in the 
$i$-th column of $M$ appears in row $j$ and the non-zero entry in the 
$j$-th column of $M$ appears in row $i$. 
We define the action of the generators of $\CC[C_r\wr R_n]$ on $ \mathcal{V}$ as follows: 
for $1\leq i\leq n-1$, set
\begin{align}\label{al:actiongel}
	 	s_{i}  M=\begin{cases}
	 		-s_{i}Ms_{i}, & \text{ if } (i,i+1)\in \pa(M);\\
	 		s_{i}Ms_{i}, & \text{ otherwise};
	 		\end{cases}
\end{align}
\begin{align}\label{al:actionPQ}
 Q  M=QMQ^{-1},	 &&	P  M=\begin{cases}
 			M, & \text{ if the first row of } M \text{ is zero};\\
 			0, & \text{ otherwise}.
 		\end{cases}
\end{align}
    
\begin{proposition}\label{prop:gel}
The module $\mathcal{V}$ is a Gelfand model for $\CC[C_r\wr R_n]$.
\end{proposition}

\begin{proof}
For $0\leq m\leq n$, let $\mathcal{V}_m$ be the $\CC$-span of 
$\mathcal{I}_m=\{M\in\mathcal{I}\mid \rn(M)=m \}$. 
From \eqref{al:actiongel} and \eqref{al:actionPQ}, we see that 
each $\mathcal{V}_m$ is closed under the action, moreover, 
it is easy to see that, directly by construction, we have 
$ \mathcal{V}_m\cong\mathcal{L}_m^{n}(\mathcal{W}_m)$.

Since $\mathcal{V}=\displaystyle{\bigoplus_{m=0}^{n}}\mathcal{V}_m$,  
the claim of the proposition follows by combining the facts that $\mathcal{W}_m$ is a 
Gelfand model for $\CC[C_r\wr S_m]$ and that 
$\displaystyle{\bigoplus_{m=0}^{n}}\mathcal{L}_{m}^n$ is an equivalence of categories,
see Lemma~\ref{lm:Requi}.
\end{proof}

\section{Jucys--Murphy elements}\label{sec:JM}

Given any subset $A$ of $[n]$, let $e_{A}$ be the diagonal matrix which has $1$ at 
the $(i,i)$-th entry for $i\in [n]\setminus A$ and zeros elsewhere. Consider the element
\begin{displaymath}
E_{A}:=\sum_{B\subseteq A} (-1)^{|B|}e_{B}.
\end{displaymath}
Then $E_{A}\in \FF[R_n]$ is an idempotent and, moreover, any two such idempotents commute with each other
(since all idempotents  of $R_n$ commute).	Assume that $r$ is non-zero in $\FF$. 
Consider the following elements in  $\FF[C_{r}\wr R_n]$:
\begin{align*}
&X_1=Q-P,  &&X_{j}=s_{j-1}X_{j-1}s_{j-1}, &\text{for } 2\leq j\leq n;\\
&Y_1=0,  &&Y_{j}=\frac{1}{r}\sum_{m=1}^{j-1} E_{\{m,j\}}\sum_{l=0}^{r-1}\xi_{m}^{l}\xi_{j}^{-l}(m,j), 
&\text{ for } 2\leq j\leq n;
\end{align*}
where $\xi_{m}^{l}\xi_{j}^{-l}\in C_r\wr S_n$ denotes the diagonal matrix with $1$'s
on the diagonal except for $\xi^l$ at the $(m,m)$-th entry and $\xi^{-l}$ at the $(j,j)$-th entry. 
Further, $(m,j)$ is the usual transposition in $S_n$. It is easy to observe that 
\begin{align}\label{al:QQ}
\xi_m^{l}\xi_j^{-l}=(1,m)Q^{l}(1,m)(1,j)Q^{-l}(1,j).
\end{align}
We will call the elements $\{X_i,Y_i\vert 1\leq i,j\leq n\}$ the {\em Jucys--Murphy elements} for $C_{r}\wr R_n$.

\begin{proposition}\label{prop:JMcommute}
For $1\leq i,j\leq n$, we have 
\begin{displaymath}
X_iX_j=X_jX_i,\quad Y_iY_j=Y_jY_i, \textrm{ and } X_iY_j=Y_jX_i.
\end{displaymath}
\end{proposition}

\begin{proof}
For $1\leq j\leq n$, let $P_{j}=e_{\{j\}}$. Note that $P=P_{\{1\}}$.
Let $Q_{1}=Q$ and, for $2\leq j\leq n$, let $Q_{j}=s_{j-1}Qs_{j-1}$. 
Then $Q_{j}$ is a diagonal matrix whose $(j,j)$-th diagonal entry is 
equal to $\xi$ and the remaining diagonal entries are equal to $1$. 
For $1\leq i\leq n$, we can write $X_{i}=Q_{i}-P_{i}$, which is a 
linear combination of diagonal matrices. Thus $X_{i}X_{j}=X_{j}X_{i}$.
 
To prove $Y_{i}Y_{j}=Y_{j}Y_{i}$, without loss of generality, we may assume that $i<j$. We can write
\begin{align}\label{al:yj}
Y_{j}=&\frac{1}{r}\bigg(\sum_{m=1}^{i-1} E_{\{m,j\}}
\sum_{l=0}^{r-1}\xi_{m}^{l}\xi_{j}^{-l}(m,j)+E_{\{i,j\}}
\sum_{l=0}^{r-1}\xi^{l}_{i}\xi^{-l}_{j}(i,j)\nonumber\\ 
&\quad \quad+\sum_{m=i+1}^{j-1} E_{\{m,j\}}\sum_{l=0}^{r-1}\xi_{m}^{l}\xi_{j}^{-l}(m,j)\bigg).
\end{align}
Now, $Y_{i}$ commutes with 
$\displaystyle{\sum_{m=i+1}^{j-1}} E_{\{m,j\}}\displaystyle{\sum_{l=0}^{r-1}}\xi_{m}^{l}\xi_{j}^{-l}(m,j)$. 
The product of $Y_i$ with the middle term of \eqref{al:yj} can be written as
\begin{align}
&Y_{i}	E_{\{i,j\}}\sum_{l=0}^{r-1}\xi^{l}_{i}\xi^{-l}_{j}(i,j)\nonumber\\ 
& \quad \quad\quad= \frac{1}{r}\bigg(\sum_{m=1}^{i-1}E_{\{m,i,j\}}
\bigg(\sum_{l=0}^{r-1}\xi_{i}^{l}\xi_j^{-l}\bigg)
\bigg(\sum_{l=0}^{r-1}\xi_{m}^{l}\xi_j^{-l}\bigg)
(m,i)(i,j)\bigg),  \label{al:first} \\
& E_{\{i,j\}}\sum_{l=0}^{r-1}\xi^{l}_{i}\xi^{-l}_{j}(i,j)Y_i\nonumber\\ 
&\quad\quad\quad=\frac{1}{r}\bigg(\sum_{m=1}^{i-1}E_{\{m,i,j\}}
\bigg(\sum_{l=0}^{r-1}\xi_{i}^{l}\xi_j^{-l}\bigg)\bigg(\sum_{l=0}^{r-1}\xi_{m}^{l}\xi_j^{-l}\bigg)
(i,j)(m,i) \bigg).\label{al:sec}
\end{align}
 
Also, we can write
\begin{align}
&Y_i\sum_{m=1}^{i-1} E_{\{m,j\}}\sum_{l=0}^{r-1}\xi_{m}^{l}\xi_{j}^{-l}(m,j)\nonumber\\
&=\frac{1}{r} \bigg(\sum_{p=1}^{i-1} E_{\{p,i\}}
\sum_{l=0}^{r-1}\xi_{p}^{l}\xi_{i}^{-l}(p,i)\bigg)
\bigg(\sum_{m=1}^{i-1} E_{\{m,j\}}\sum_{l=0}^{r-1}\xi_{m}^{l}\xi_{j}^{-l}(m,j)\bigg)\nonumber\\
&=\frac{1}{r}\bigg(\sum_{p=1}^{i-1} E_{\{p,i\}}
\sum_{l=0}^{r-1}\xi_{p}^{l}\xi_{i}^{-l}(p,i)\bigg)
\bigg(\sum_{m=1,m\neq p}^{i-1} E_{\{m,j\}}\sum_{l=0}^{r-1}\xi_{m}^{l}\xi_{j}^{-l}(m,j)\bigg)\nonumber\\
&\quad \quad +\frac{1}{r}\bigg(\sum_{p=1}^{i-1} E_{\{p,i\}}
\sum_{l=0}^{r-1}\xi_{p}^{l}\xi_{i}^{-l}(p,i)\bigg)
\bigg( E_{\{p,j\}}\sum_{l=0}^{r-1}\xi_{p}^{l}\xi_{j}^{-l}(p,j)\bigg).\label{al:mid}
\end{align}
 
In the above, we note that elements in the first term commutes with each other 
and the second term simplifies to 
\begin{align*}
\sum_{p=1}^{i-1}E_{\{p,i,j\}}\bigg(\sum_{l=0}^{r-1}\xi_{p}^{l}\xi_i^{-l}\bigg)
\bigg(\sum_{l=0}^{r-1}\xi_{i}^{l}\xi_j^{-l}\bigg)
(p,i)(p,j),
\end{align*}
which is equal to \eqref{al:sec} using, for $1\leq p\leq  i-1$, that $(p,i)(p,j)=(i,j)(p,i)$ and
\begin{displaymath}
\bigg(\sum_{l=0}^{r-1}\xi_{p}^{l}\xi_i^{-l}\bigg)
\bigg(\sum_{l=0}^{r-1}\xi_{i}^{l}\xi_j^{-l}\bigg)=
\bigg(\sum_{l=0}^{r-1}\xi_{i}^{l}\xi_j^{-l}\bigg)\bigg(\sum_{l=0}^{r-1}\xi_{p}^{l}\xi_j^{-l}\bigg).
\end{displaymath}
Similarly to \eqref{al:mid}, we can write
\begin{align*}
&\sum_{m=1}^{i-1} E_{\{m,j\}}\sum_{l=0}^{r-1}\xi_{m}^{l}\xi_{j}^{-l}(m,j)Y_i\nonumber\\
&=\frac{1}{r}\bigg(\sum_{m=1}^{i-1} E_{\{m,j\}}\sum_{l=0}^{r-1}\xi_{m}^{l}\xi_{j}^{-l}(m,j)\bigg)
\bigg(\sum_{p=1, p\neq i}^{i-1} E_{\{p,i\}}\sum_{l=0}^{r-1}\xi_{p}^{l}\xi_{i}^{-l}(p,i)\bigg)\nonumber\\
&\quad \quad +\frac{1}{r}\bigg(\sum_{m=1}^{i-1} E_{\{m,j\}}\sum_{l=0}^{r-1}\xi_{m}^{l}\xi_{j}^{-l}(m,j)\bigg)
\bigg( E_{\{m,i\}}\sum_{l=0}^{r-1}\xi_{m}^{l}\xi_{i}^{-l}(m,i)\bigg),
\end{align*}
where the elements in the first term commutes with each other and the second term is equal 
to \eqref{al:first}. All of the  above, finally, yield that $Y_iY_j=Y_jY_i$.

If $i>j$, then $X_i=Q_i-P_i$ commutes with each term in $Y_j$, which implies that $X_iY_j=Y_jX_i$. 
Let us assume that $i\leq j$. Both $Q_i$ and $P_i$ commute with every element in the 
first and the last terms in \eqref{al:yj}. Now $Q_i$ commutes with
$E_{\{i,j\}}\displaystyle{\sum_{l=0}^{r-1}}\xi^{l}_{i}\xi^{-l}_{j}(i,j)$ because 
$\displaystyle{\sum_{l=0}^{r-1}}\xi_{i}^{l+1}\xi^{-l}_{j}=\displaystyle{\sum_{l=0}^{r-1}}\xi_{i}^{l}\xi^{-l+1}_{j}$. 
Furthermore, $P_i$ commutes with $E_{\{i,j\}}\displaystyle{\sum_{l=0}^{r-1}}\xi^{l}_{i}\xi^{-l}_{j}(i,j)$  
because $P_{i}E_{\{i,j\}}=0=E_{\{i,j\}}P_j$. This completes the proof.
\end{proof}

\begin{theorem}\label{thm:JM}
For $\lambda^{(r)}\in \Lambda_{r}(\leq n)$ and  $L=(L_{1},\ldots,L_{r})\in\mathcal{Y}(\lambda^{(r)}, n)$, let $v_{L}$ be the corresponding Gelfand--Zeitlin basis vector of $V_{\lambda^{(r)}}^n$. For $1\leq i\leq n$, we have
\begin{align}
&X_{i}  v_{L}=	\begin{cases}\label{ex:X}
\xi^{k-1} v_{L}, & \text{ if } i\in L \text{ and }\sgn_{L}(i)=\xi^{k-1};\\
0, & \text{ otherwise; }
\end{cases}\\
&	Y_{i}  v_{L}=	\begin{cases}\label{ex:Y}
\ct(L(i)) v_{L}, & \text{ if } i\in L;\\
0, & \text{ otherwise. }
\end{cases}
\end{align}
\end{theorem}

\begin{proof}
We begin with proving $\eqref{ex:X}$ using induction on $i$. 

From $\eqref{al:P-Q}$, we have:
\begin{align*}
X_1  v_{L}=Q  v_{L}-P  v_{L}=\begin{cases}
\xi^{k-1}v_{L}, & \text{ if } 1\in L \text{ and }\sgn_{L}(1)=\xi^{k-1};\\
0, & \text{ otherwise.} 
\end{cases}
\end{align*}
Assume now that \eqref{ex:X} is true for $2\leq i< m\leq n$, and let us prove it for $i=m$. 
By definition, $X_m=s_{m-1}X_{m-1}s_{m-1}$. We need to consider several cases.

{\bf Case~1.} $m-1\in L,m\notin L$. Then $s_{m-1}  v_L=v_{s_{m-1}L}$ and $s_{m-1}L$ 
does not contain $m-1$. From the inductive assumption,  we have $X_{m-1}  v_{s_{m-1}L}=0$ 
and this  implies $X_{m}  v_{L}=0$.

{\bf Case~2.} $m-1\notin L$, $m\in L $. Let us further assume that $\sgn_{L}(m)=\xi^{k-1}$, 
for some $1\leq k\leq r$. Then $s_{m-1}  v_{L}=v_{s_{m-1}L}$ and $\sgn_{s_{m-1}L}(m-1)=\xi^{k-1}$.
Applying the inductive assumption, we get the desired formula.

{\bf Case~3.} $m-1\in L$, $m\in L$ and $\sgn_{L}(m-1)\neq\sgn_{L}(m)$. This case is analogous to Case~2.

{\bf Case~4.} $m-1\in L$, $m\in L$ and $\sgn_{L}(m-1)=\sgn_{L}(m)=\xi^{k-1}$. 
In this case, $m-1\in L_k$ and $m\in L_k$, for some $1\leq k\leq r$. Then 
$$s_{m-1}  v_{L}=a_{L}(m-1)v_{L}+(1+a_{L}(m-1))v_{s_{m-1}L}.$$
Now we have to consider two subcases.

{\bf Subcase~4.1.} $s_{m-1}L\notin\mathcal{Y}(\lambda^{(r)},n)$. 
Then $m-1$ and $m$ appear adjacent to each other either in the same row 
of $L_k$ or in the same column of $L_k$. This implies that $a_{L}(m-1)=\pm 1$. 
Also, recall that, by convention, in this case we have $v_{s_{m-1}L}=0$. 
Now, applying the inductive assumption, we obtain 
$$s_{m-1}X_{m-1}s_{m-1}  v_{L}=a_{L}^{2}(m-1)\xi^{k-1}v_{L}=\xi^{k-1}v_{L}.$$

{\bf Subcase~4.2.} $s_{m-1}L\in\mathcal{Y}(\lambda^{(r)},n)$. Using the inductive assumption, we have:

\resizebox{\textwidth}{!}{
$
\begin{array}{rl}
s_{m-1}X_{m-1}s_{m-1}  v_{L}&=
\xi^{k-1}\big(a_{L}(m-1)s_{m-1}v_{L}+(1+a_{L}(m-1))s_{m-1}  v_{s_{m-1}L}\big)\\
&=\xi^{k-1}a_{L}(m-1)\big(a_{L}(m-1)v_{L}+(1+a_{L}(m-1))v_{s_{m-1}L}\big)+\\
&\quad\,\,\xi^{k-1}(1+a_{L}(m-1))\big(-a_{L}(m-1)v_{s_{m-1}L}+(1-a_{L}(m-1))v_{L}\big)\\
&=\xi^{k-1}\big(A v_{L}+Bv_{s_{m-1}L}\big), 
\end{array}
$
}
where
\begin{align*}
&A=a_{L}^2(m-1)+(1+a_{L}(m-1))(1-a_{L}(m-1))=1\\
&B=a_{L}(m-1)(1+a_{L}(m-1))-(1+a_{L}(m-1))a_{L}(m-1)=0.
\end{align*}
This proves \eqref{ex:X}.

To prove \eqref{ex:Y}, we again use induction on $i$. For $i=1$, the claim is obvious. 
Assume now that $\eqref{ex:Y}$ is true for $2\leq i< m\leq n$, and let us  prove it for $i=m$.

We have 
\begin{displaymath}
Y_m=s_{m-1}Y_{m-1}s_{m-1}+  \frac{1}{r}E_{\{m-1,m\}}
\left(\sum_{l=0}^{r-1}\xi_{m-1}^{l}\xi_{m}^{-l}\right)s_{m-1}.
\end{displaymath}
We need to consider several cases. Note that 
in all the cases below, the action of $\xi_{m-1}^{l}\xi_{m}^{-l}$ is computed using \eqref{al:rooksi}, \eqref{al:P-Q} and \eqref{al:QQ}; in addition, we also use
\begin{align*}
	E_{\{m-1,m\}}v_{L}=\begin{cases}
		v_{L}, &  \text{ if } m-1\in L, m\in L;\\
		0, & \text{ otherwise}.		
		\end{cases}
	\end{align*}

{\bf Case~1.} $m-1\in L,m\notin L$. Then $s_{m-1}  v_L=v_{s_{m-1}L}$ and $s_{m-1}L$ does not contain $m-1$. 
Further, assume that $\sgn_{L}(m-1)=\xi^{k-1}$, so that $\sgn_{s_{m-1}L}(m)=\xi^{k-1}$.
By the inductive assumption, $s_{m-1}Y_{m-1}s_{m-1}  v_{L}=s_{m-1}Y_{m-1}  v_{s_{m-1}L}=0$. Also,
\begin{align*}
\frac{1}{r}E_{\{m-1,m\}}\left(\sum_{l=0}^{r-1}\xi_{m-1}^{l}\xi_{m}^{-l}\right)s_{m-1}  v_{L}
&=\frac{1}{r}E_{\{m-1,m\}}\left(\sum_{l=0}^{r-1}\xi_{m-1}^{l}\xi_{m}^{-l}\right)v_{s_{m-1}L}\\
&=\frac{1}{r}E_{\{m-1,m\}}\left(\sum_{l=0}^{r-1}\xi^{-l(k-1)}\right)  v_{s_{m-1}L}\\
&=\frac{1}{r}\left(\sum_{l=0}^{r-1}\xi^{-l(k-1)}\right)E_{\{m-1,m\}}  v_{s_{m-1}L}=0.
\end{align*}

{\bf Case~2.} $m-1\notin L,m\in L$. Then $s_{m-1}  v_L=v_{s_{m-1}L}$. 
Further, assume that $\sgn_{L}(m)=\xi^{k-1}$, so that $\sgn_{s_{m-1}L}(m-1)=\xi^{k-1}$. 
By the inductive assumption, we have 
$s_{m-1}Y_{m-1}s_{m-1}  v_{L}=\ct((s_{m-1}L)(m-1))v_{L}=\ct(L(m))v_{L}$. 
Similarly to Case~1, the term 
$\displaystyle\frac{1}{r}E_{\{m-1,m\}}\left(\sum_{l=0}^{r-1}\xi_{m-1}^{l}\xi_{m}^{-l}\right)s_{m-1}$ acts as 
zero on $v_{L}$.

{\bf Case~3.} $m-1\in L$, $m\in L$. Assume further that 
$\sgn_{L}(m-1)=\xi^{k_1-1}$ and $\sgn_{L}(m)=\xi^{k_2-1}$, where $1\leq k_1\neq k_2\leq r$. 
By the inductive assumption, we have 
\begin{align*}
s_{m-1}Y_{m-1}s_{m-1}  v_{L}&=\ct((s_{m-1}L)(m-1))v_{L}=\ct(L(m))v_{L}
\end{align*}
and
\begin{align*}
\frac{1}{r}E_{\{m-1,m\}}\left(\sum_{l=0}^{r-1}\xi_{m-1}^{l}\xi^{-l}_{m}\right)s_{m-1}
  v_{L}=\frac{1}{r}\left(\sum_{l=0}^{r-1}\xi^{l(k_2-k_1)}\right)v_{s_{m-1}L}=0.
\end{align*}
The last equality in the above is a consequence of the following: for an integer $s$, we have
\begin{align*}\label{al:xisum}
\sum_{l=0}^{r-1}\xi^{ls}=
\begin{cases}
0, & \text{ if } s\neq 0;\\
r, & \text{ if } s=0.
\end{cases}
\end{align*}

{\bf Case~4.} $m-1\in L$, $m\in L$ and $\sgn_{L}(m-1)=\sgn_{L}(m)=\xi^{k-1}$. 
This means that $m-1\in L_k$ and $m\in L_k$, for some $1\leq k\leq r$. Then we have 
$$s_{m-1}  v_{L}=a_{L}(m-1)v_{L}+(1+a_{L}(m-1))v_{s_{m-1}L}.$$
We have now to consider two subcases.

{\bf Subcase~4.1.} $s_{m-1}L\notin\mathcal{Y}(\lambda^{(r)},n)$. 
Then $m-1$ and $m$ appear adjacent to each other either in a same row of $L_k$ 
or in a same column of $L_k$, This implies that $a_{L}(m-1)=\pm 1$. 
Also, recall our convention that $v_{s_{m-1}L}=0$ in this case. Now, by the 
inductive assumption, we obtain
$$s_{m-1}Y_{m-1}s_{m-1}  v_{L}=(a_{L}(m-1))^{2}\ct(L(m-1))v_{L}=\ct(L(m-1))v_{L}.$$

Further, we have:
\begin{align*}
\frac{1}{r}E_{\{m-1,m\}}\left(\sum_{l=0}^{r-1}\xi_{m-1}^{l}\xi_{m}^{-l}\right)s_{m-1}  v_{L}
&= \frac{1}{r}E_{\{m-1,m\}}\left(\sum_{l=0}^{r-1}\xi_{m-1}^{l}\xi_{m}^{-l}\right)  (a_{L}(m-1)v_{L})\\
&= \frac{1}{r}E_{\{m-1,m\}} \left(\sum_{l=0}^{r-1}\xi^{l(k-k)}\right)(a_{L}(m-1)v_{L})\\
&= \frac{1}{r}E_{\{m-1,m\}}  \left(\sum_{l=0}^{r-1}1\right)a_{L}(m-1)v_{L}\\
&=a_{L}(m-1)v_{L}.
\end{align*}
Since $a_{L}(m-1)+\ct(L(m-1))=\ct(L(m))$, we get the desired answer.

{\bf Subcase~4.2.} $s_{m-1}L\in\mathcal{Y}(\lambda^{(r)},n)$. 
Using the inductive assumption, we have:
\begin{align}
&s_{m-1}Y_{m-1}s_{m-1}  v_{L}\nonumber\\
&= a_{L}(m-1)\ct(L(m-1))s_{m-1}  v_{L}+(1+a_{L}(m-1))
\ct(L(m))s_{m-1}  v_{s_{m-1}L}\nonumber\\
&= a_{L}(m-1)\ct(L(m-1))\big(a_{L}(m-1)v_{L}+(1+a_{L}(m-1))v_{s_{m-1}L}\big)\nonumber\\
& \quad\,\,+(1+a_{L}(m-1))\ct(L(m))
\big(-a_{L}(m-1)v_{s_{m-1}L}+(1-a_{L}(m-1))v_{L}\big).\nonumber
\end{align}
Now, we compute the action of second term on $v_L$:

\resizebox{\textwidth}{!}{
$
\begin{array}{l}
\displaystyle
\frac{1}{r}E_{\{m-1,m\}}\left(\sum_{l=0}^{r-1}\xi_{m-1}^{l}\xi_{m}^{-l}\right)s_{m-1}
  v_{L}\nonumber\\ \displaystyle
= \frac{1}{r}E_{\{m-1,m\}}
\bigg(\left(\sum_{l=0}^{r-1}\xi_{m-1}^{l}\xi_{m}^{-l}\right)
  a_{L}(m-1)v_{L} +\left(\sum_{l=0}^{r-1}\xi_{m-1}^{l}\xi_{m}^{-l}\right)
 (1+a_{L}(m-1))v_{s_{m-1}L}\bigg)\nonumber\\ \displaystyle
=\frac{1}{r}E_{\{m-1,m\}} \bigg(\left(\sum_{l=0}^{r-1}\xi^{l(k-k)}\right) 
a_{L}(m-1)v_{L} +\left(\sum_{l=0}^{r-1}\xi^{l(k-k)}\right)
(1+a_{L}(m-1))v_{s_{m-1}L}\bigg)\nonumber\\ \displaystyle
=\frac{1}{r}E_{\{m-1,m\}} \bigg(\left(\sum_{l=0}^{r-1}1\right) 
a_{L}(m-1)v_{L} +\left(\sum_{l=0}^{r-1}1\right)
(1+a_{L}(m-1))v_{s_{m-1}L}\bigg)\nonumber\\ \displaystyle
=a_{L}(m-1)v_{L}+(1+a_{L}(m-1))v_{s_{m-1}L}.
\end{array}
$
}

Combining the coefficients at $v_{L}$ and $v_{s_{m-1}L}$, we get
\begin{displaymath}
Y_{m}  v_{L}=Cv_{L}+Dv_{s_{m-1}L},
\end{displaymath}
where
\begin{align*}
C&=a_{L}^{2}(m-1)\ct(L(m-1))+(1-a_{L}^{2}(m-1))\ct(L(m))+a_{L}(m-1)\\
&=a_{L}(m-1)\big[a_{L}(m-1)(\ct(L(m-1))-\ct(L(m)))+1\big]+\ct(L(m))\\
&=a_{L}(m-1)[0]+\ct(L(m))\\
&=\ct(L(m))
\end{align*}
and
\begin{align*}
D&=a_{L}(m-1)\ct(L(m-1))(1+a_{L}(m-1))-\\
&\quad\,-(1+a_{L}(m-1))\ct(L(m))a_{L}(m-1)+(1+a_{L}(m-1))\\
&=(1+a_{L}(m-1))\big[a_{L}(m-1)(\ct(L(m-1))-\ct(L(m)))+1\big]\\
&=0.
\end{align*}
This implies \eqref{ex:Y} and completes the proof of the theorem.
\end{proof}

\begin{remark}
For $r=1$, the corresponding Jucys--Murphy elements were given 
in \cite[Equation~2.9]{MS21}. It is easy to show, using induction, 
that the elements in \cite{MS21} and the elements in \eqref{ex:Y} 
are the same.
\end{remark}

As an immediate consequence of Theorem~\ref{thm:JM}, we have the following:

\begin{corollary}\label{coro:irre}
The eigenvalues of  the action of  $X_i$ and $Y_i$, 
for $1\leq i\leq n$, on Gelfand--Zeitlin basis vectors  distinguish 
the latter and, consequently, the isomorphism classes of simple 
$\CC[C_r\wr R_n]$-modules.
\end{corollary}

In particular, it follows that the Gelfand--Zeitlin subalgebra 
of $\CC[C_r\wr R_n]$ is generated by Jucys--Murphy elements.

\section{Bicyclic monoid}\label{sec:bcyclic}

The results of this section should be known to experts. 
However, we could not trace an explicit reference, so we provide
all proofs, for completeness.	

Recall, cf. \cite{ClPr}, that the bicyclic monoid $\mathcal{B}$ is an infinite 
monoid generated by two elements $a$ and $b$ and given by the following presentation:
\begin{displaymath}
\mathcal{B}:=\langle a, b\mid ab=1 \rangle.
\end{displaymath}
We have $\mathcal{B}=\{b^{n_1}a^{n_2}\mid n_1,n_2\in\NN\}$, where
$\NN:=\{0,1,2,\ldots\}$. 

Consider the $\CC$-vector space $V_{\NN}$ with $\NN$ as a basis. 
Define the actions of $b$ and $a$ on $V_{\NN}$  as follows:
\begin{align*}
b  i= i+1,\qquad\text{ and } \qquad a  i=\begin{cases}
i-1 ,& \text{ if } i>0;\\
0, & \text{ otherwise.}
\end{cases}
\end{align*}
In the above, note that $a$ acts as zero on the basis vector $0\in\NN$. Then it is easy to check that $V_{\NN}$ becomes a simple $\CC[\mathcal{B}]$-module.

For a non-zero $\lambda\in\CC$, let $V_{\lambda}$ be the $1$-dimensional
$\mathbb{C}$-vector space with basis $v_\lambda$. 
Define the actions of $b$ and $a$ on $V_{\lambda}$ as follows:
\begin{displaymath}
b  v_\lambda=\lambda v_\lambda\quad\text{ and }
\quad a  v_\lambda=\lambda^{-1}v_\lambda.
\end{displaymath}
Then $V_{\lambda}$ is a simple $\CC[\mathcal{B}]$-module.

\begin{proposition}\label{thm:bcyclic}
Let $V$ be a simple $\CC[\mathcal{B}]$-module. 
Then either $V\cong V_{\NN}$ or $V\cong V_{\lambda}$, 
for some non-zero $\lambda\in\CC$.
\end{proposition}

\begin{proof}
Let $L_a$ and $L_b$ be linear operators on $V$ representing the 
actions of $a$ and $b$. Since $ab=1$, we have $L_a\circ L_b=\mathrm{Id}_{V}$,
in particular, $L_a$ is surjective and $L_b$ is injective.
We need to consider two cases.

{\bf Case 1.} Suppose $L_a$ is injective and hence invertible. 
Then $L_b=L_a^{-1}$ and, in particular, $L_aL_b=L_bL_a$
and hence $L_a$ commutes with the action of $\CC[\mathcal{B}]$.
From Schur--Dixmier Lemma, it then follows that $L_a=\lambda \mathrm{Id}_{V}$,
for some $\lambda\in\CC$. Obviously, this $\lambda$ must be non-zero.
Consequently, $L_b=\lambda^{-1} \mathrm{Id}_{V}$. In this case
any subspace of $V$ is, clearly, a submodule. Therefore $V$ must 
have dimension one by simplicity and hence is isomorphic to $V_{\lambda}$.

{\bf Case 2.} Suppose $L_a$ is not injective. 
Then there exists a non-zero $v\in V$ such that $a  v=0$. 
Consider the subspace $W$ spanned by $\{b^n  v\mid n\in\NN\}$. 
Since $L_b$ is injective, the set $\{b^n  v\mid n\in\NN\}$ is 
linearly independent and, therefore, $W$ is infinite-dimensional. 
Clearly, $W$ is stable under the action of $\CC[\mathcal{B}]$ 
and is isomorphic to $V_{\NN}$. Since $V$ is simple, we must have $V=W$
and, finally, $V\cong V_{\NN}$.
\end{proof}

\section{Grothendieck groups for rook monoid algebras}\label{sec:gro}

In this section, we assume that the characteristic of $\FF$ is $p>0$.   
We identify  the prime subfield $\FF_p$ of $\FF$ with the additive 
cyclic group of order $p$. We start by recalling the classical results
related to the tower of symmetric groups.

\subsection{The case of symmetric groups}

We have the following {\em Jucys--Murphy elements} for the algebra $\FF[S_n]$:
\begin{align}\label{al:JMS}
\widetilde{Y}_k=\sum_{i=1}^{k-1}(i,k),\quad \text{where }  k\in [n].
\end{align}
Here, $(i,k)\in S_n$ is a transposition.

For $V\in \FF[S_n]\Mod$,  the eigenvalues of the operator $\widetilde{Y}_k$ on 
$V$ lie in $\FF_p$ (e.g., see \cite[Lemma~2.2]{KB}).  Since the elements in \eqref{al:JMS} commute with each other, for 
$\underline{r}=(r_1,r_2,\ldots,r_n)\in \FF_p^{n}$, the common generalized 
eigenspace of $V$ with respect to the elements in \eqref{al:JMS} is
\begin{displaymath}
V_{\underline{r}}:=\{v\in V\mid (\widetilde{Y}_k-r_k)^{N}v=0,   
\text{  for all } k\in [n] \text{  and }N\gg0 \}.
\end{displaymath}
We then have the decomposition 
$V=\displaystyle{\bigoplus_{\underline{r}\in \FF_p^n}} V_{\underline{r}}$.

For $\underline{r}\in \FF_p^n$ and  $l\in \FF_{p}$, let  
$\mu_{l}:=\vert\{k\in[n]\mid r_k=l\}\vert$. 
Then $\displaystyle{\sum_{l\in \FF_p}}\mu_l=n$ and the tuple 
$\ty(\underline{r})=(\mu_0,\ldots,\mu_{p-1})\in \NN^{p}$ is called the {\em weight} 
of $\underline{r}$.   For 
$$\mu\in\mathcal{T}_n:=\{(\mu_0,\ldots,\mu_{p-1})\in \NN^{p}\mid \displaystyle{\sum_{r=0}^{p-1}\mu_r}= n\},$$ 
let 
\begin{align}\label{al:symdecom}
V[\mu]=\displaystyle{\bigoplus_{\substack{\underline{r}\in \FF_p^n\\\ty(\underline{r})=\mu}}}V_{\underline{r}}.
\end{align}
The following lemma is \cite[Lemma~2.4]{KB} and the key point in their proof is that the center of $\FF[S_n]$ is generated by the elementary symmetric polynomials (see \cite[Section 2]{Mac}) in $\{Y_k\mid k\in[n]\}$, which can be found in \cite{FHG} and \cite{Jucys}.

\begin{lemma}\label{lm:Sblock}
For $V\in \FF[S_n]\Mod$ and $\mu\in\mathcal{T}_n$, 
the space $V[\mu]$ is a $\FF[S_n]$-module, moreover,
$V\cong \displaystyle{\bigoplus_{\mu \in\mathcal{T}_n}}V[\mu]$.
\end{lemma}

\subsubsection{Decompositions of induction and restriction functors}

For $i\in \FF_p$ and $ \gamma=(\gamma_0,\ldots,\gamma_{p-1})\in \mathcal{T}_n$, 
denote  by $\gamma+i$ the element in $\NN^p$ whose $i$-th coordinate is 
$\gamma_i+1$ and the remaining coordinates are the same as those of $\gamma$. 
Similarly, if $\gamma_i\neq 0$, denote by $\gamma-i$  an element in $\NN^p$,  
whose $i$-th coordinate is $\gamma_i-1$ and the remaining coordinates are 
the same as those of $\gamma$.

For  $\gamma\in\mathcal{T}_n$ and $V=V[\gamma]\in \FF[S_n]\Mod$, 
define the functors
\begin{displaymath}
\widetilde{\res}_i:\FF[S_n]\Mod\to \FF[S_{n-1}]\Mod\quad\text{ and } 
\quad\widetilde{\ind}_i:\FF[S_n]\Mod\to \FF[S_{n+1}]\Mod
\end{displaymath}
as follows:
\begin{align*}
&\widetilde{\res}_i(V[\gamma]):=\begin{cases}
(\Res^{\FF[S_n]}_{\FF[S_{n-1}]}V[\gamma])[\gamma-i], & \text{ if } \gamma_i\neq 0; \\
0, &\text{ otherwise};
\end{cases}\\
&\widetilde{\ind}_i(V[\gamma]):=
(\Ind^{\FF[S_{n+1}]}_{\FF[S_{n}]}V[\gamma])[\gamma+i].
\end{align*}
This definition extends to any object in $\FF[S_n]\Mod$
and hence completely defines $\widetilde{\res}_i$
and $\widetilde{\ind}_i$ due to Lemma~\ref{lm:Sblock}.  
The functors $\widetilde{\res}_i$ and $\widetilde{\ind}_i$ 
are called the {\em $i$-restriction} and the {\em $i$-induction} functors, 
respectively. Using these definitions, we get the following decomposition 
of the restriction and induction functors in terms of the 
$i$-restriction and $i$-induction functors
\begin{displaymath}
\Res^{\FF[S_n]}_{\FF[S_{n-1}]}(V)\cong \bigoplus_{i\in \FF_p}
\widetilde{\res}_i \quad \text{and} \quad 
\Ind^{\FF[S_{n+1}]}_{\FF[S_{n}]}(V)\cong \bigoplus_{i\in \FF_p}
\widetilde{\ind}_i. 
\end{displaymath}

Recall that for a finite-dimensional algebra $A$, $G_0(A)$ denotes the complexified Grothendieck group of $A\Mod$.  For $M\in A\Mod$, let $[M]\in G_{0}(\FF[S_n])$ 
the class of the corresponding of $M$ in $G_0(A)$
 If $B$ is another finite-dimensional algebra over $\FF$ and $F$ is a covariant exact functor from $A\Mod$ to $B\Mod$, then let $[F]$ denote the induced $\CC$-linear map from $G_0(A)$ to $G_0(B)$.

The functors $\widetilde{\res}_i$ and $\widetilde{\ind}_i$ 
are exact and hence they induce the following linear maps 
\begin{align*}
	[\widetilde{\res}_i]:G_0(\FF[S_n])\to G_0(\FF[S_{n-1}]) \text{ and }
	 [\widetilde{\ind}_i]:G_0(\FF[S_n])\to G_0(\FF[S_{n+1}]).
\end{align*}
For $\lambda\in\Lambda^{p}(n)$, we denote by $D^{\lambda}$
 the corresponding simple module of $\FF[S_n]\Mod$. 
Then the set $\{[D^{\lambda}]\mid \lambda\in\Lambda^{p}(n)\}$ 
is a basis of $G_{0}(\FF[S_n])$. In order to describe 
actions of $[\widetilde{\res}_i]$ and $[\widetilde{\ind}_i]$ 
in this basis, we need to recall some terminology.

\begin{definition}
Let $\lambda$ be a partition and $i\in \FF_p$.
\begin{enumerate}
\item For a box $b\in[\lambda]$, the {\em residue of $b$}, 
denoted by $\res(b)$, is defined as the content of $b$ modulo $p$.
\item A box $b\in[\lambda]$ is called {\em $i$-removable} if $\res(b)=i$ 
and $[\lambda]\setminus \{b\}$ is a Young diagram.
\item A box $b$ adjacent to $[\lambda]$ is called {\em $i$-addable}
for $\lambda$ if $\res(b)=i$ and $[\lambda]\bigcup\{b\}$ is a Young diagram.
\item Let us label all $i$-addable boxes for $\lambda$ by $+$ and 
all $i$-removable boxes in $[\lambda]$ by $-$. Then, going along the
rim of $[\lambda]$ from bottom left to top right, we can read off
a sequence of $+$ and $-$. This sequence is called the {\em $i$-signature} 
of $\lambda$. We refer to \cite[Subsection~2.4]{KB} for more details
and an example.
\item  The {\em reduced $i$-signature of $\lambda$} is obtained from the
$i$-signature of $\lambda$ by removing, recursively, all adjacent pairs 
of the form $-+$. Note that the reduced $i$-signature is a 
sequence of $+$'s followed by a sequence of $-$'s. The boxes labeled by 
$-$ (respectively, $+$) in the reduced $i$-signature of $\lambda$ 
are called {\em $i$-normal} (respectively, {\em $i$-conormal}). 
We again refer to \cite[Subsection~2.4]{KB} for more details
and an example.
\end{enumerate} 	
\end{definition}

Finally, we have the following result describing  actions of 
the linear operators $[\widetilde{\res}_i]$  and $[\widetilde{\ind}_i]$
on the classes of simple modules, 
see \cite[Theorems $E$ and $E'$]{KB1}.

\begin{theorem}\label{thm:action}
For $\lambda\in\Lambda^{p}(n)$, we have 
\begin{align*}
[\widetilde{\res}_i]([D^\lambda])=
\sum_{\mu\in\Lambda^{p}(n-1)}\alpha_{\lambda\mu}[D^\mu]  \quad\text{ and }\quad
[\widetilde{\ind}_i]([D^\lambda])=
\sum_{\nu\in\Lambda^{p}(n+1)}\beta_{\lambda\nu}[D^\nu],
\end{align*}
where 
\begin{enumerate}[$($a$)$]
\item $\alpha_{\lambda\mu}=0$ unless there exists an $i$-normal box 
$b\in[\lambda]$ such that $[\lambda]\setminus\{b\}=[\mu]$, in which case 
$\alpha_{\lambda\mu}$ is the number of $i$-normal boxes in the $i$-reduced signature of $[\lambda]$ that are to the right to $b$ (including $b$);
\item $\beta_{\lambda\nu}=0$ unless there exists an $i$-conormal box $b$ adjacent to 
$[\lambda]$ such that $[\lambda]\bigcup\{b\}=[\nu]$, in which case  
$\beta_{\lambda\nu}$ is the number of $i$-conormal boxes in the $i$-reduced signature of $[\lambda]$ that are to the left to $b$ (including $b$).
\end{enumerate}
\end{theorem}

For every $i\in \FF_p$, we can view both $[\widetilde{\res}_i]$ and 
$[\widetilde{\ind}_i]$ as $\mathbb{C}$-linear endomorphisms of the space
$$\tilde{\mathcal{G}}_{\CC}=\displaystyle{\bigoplus_{n\in\NN}}G_0(\FF[S_n]).$$ 
Let $\widehat{\mathfrak{sl}}_{p}(\CC)$ denote the affine Lie algebra of type $A_{p-1}^{(1)}$ 
with the Chevalley generators $e_{i}, f_{i}$, for $i\in \FF_p$. Also, let 
$V(\Lambda_{0})$ denote the basic representation of $\widehat{\mathfrak{sl}}_{p}(\CC)$ with dominant integral weight $\Lambda_0$, see \cite[Section 4]{LLT}, \cite[Chapter 11]{K3},   and also \cite{Kacbook, Kac}.
The following theorem, proved in \cite{LLT} and \cite{Groj}, relates 
$\tilde{\mathcal{G}}_{\CC}$ to $V(\Lambda_{0})$ (see also \cite[Theorem 2.7]{KB}).

\begin{theorem}\label{thm:Smod}
{\hspace{1mm}}

\begin{enumerate}[$($a$)$]
\item\label{thm:Smod.1} 
For $i\in \FF_p$, the endomorphisms $[\widetilde{\res}_i]$ 
and $[\widetilde{\ind}_i]$ of $\tilde{\mathcal{G}}_{\CC}$ 
satisfy the defining relations for the Chevalley generators of $\widehat{\mathfrak{sl}}_{p}(\CC)$.
\item \label{thm:Smod.2} 
For $i\in \FF_p$, let $e_{i}$ and $f_{i}$ act on $\tilde{\mathcal{G}}_{\CC}$
via $[\widetilde{\res}_i]$ and $[\widetilde{\ind}_i]$, respectively. 
Then the obtained $\widehat{\mathfrak{sl}}_{p}(\CC)$-module is isomorphic to
the basic representation $V(\Lambda_{0})$.
\end{enumerate}
\end{theorem}

\subsection{The case of rook monoids}

Recall Jucys--Murphy elements of $\mathbb{Z}[R_n]$:
\begin{align*}
X_{j}:=E_{\{k\}}\quad\text{and}\quad
Y_{j}:= \sum_{i=1}^{j-1}E_{\{i,j\}}(i,j),\quad \text{ where }  j\in [n].
\end{align*}

\begin{lemma} 
The idempotent $E_{\{n-1,n\}}$ commutes with $Y_{n-1}$, $Y_{n}$ and $s_{n-1}$.
\end{lemma}

\begin{proof}
This is a straightforward computation. 
\end{proof}

From Theorem~\ref{thm:JM}, we have that the eigenvalues of the actions of 
$X_i$ and $Y_i$ on all elements of the Gelfand--Zeitlin basis of simple
$\mathbb{C}[R_n]$-modules are integers. It follows that the eigenvalues 
of the actions of  $X_i$ and $Y_i$ on all $\mathbb{C}[R_n]$-modules are integers.
The following lemma describes an analogue of the latter statement over
$\Bbbk$.

\begin{lemma}\label{lem6.6}
Let $M\in \FF[R_n]\Mod$. Then, for $1\leq i\leq n$, the eigenvalues 
of the linear operators $X_i,Y_i\in\End_{\FF}(M)$ lie in $\FF_p$.
\end{lemma}

\begin{proof}
As we just mentioned, from Theorem~\ref{thm:JM} we know that the
eigenvalues of the actions of $X_i$ and $Y_i$ on all $\mathbb{C}[R_n]$-modules 
are integers. In particular, this applies to the regular
$\mathbb{C}[R_n]$-module. This means that the characteristic polynomial
of $X_i$ (or $Y_i$) on the regular $\mathbb{C}[R_n]$-module has integral roots
over $\mathbb{C}$. 

Since Jucys--Murphy elements are in $\mathbb{Z}[R_n]$,
over the field $\Bbbk$ we just need to reduce the above characteristic polynomial
modulo $p$. Since the original polynomial had integral roots,
the reduction modulo $p$ will have roots in $\Bbbk_p$. The claim follows.
\end{proof}	

As an immediate consequence of Lemma~\ref{lem6.6}, we have:

\begin{corollary}
For $M\in \FF[R_n]\Mod$,  we have $M\cong \displaystyle{\bigoplus_{\substack{\underline{i}\in \FF_{p}^n\\ 
\underline{j}\in\{0,1\}^n}}}M_{\underline{i},\underline{j}},$
where 
\begin{displaymath}
M_{\underline{i},\underline{j}}=\{m\in M\mid (X_k-j_k)^Nm=0, 
(Y_k-i_k)^Nm=0, \textrm{ for } k\in[n]\text{ and }N\gg 0\}.
\end{displaymath}
\end{corollary}

For $\underline{i}\in \FF_p^n$, 
$\underline{j}\in\{0,1\}^n$ and $r\in \FF_p$, define
\begin{displaymath}
	\gamma_{r}:=\vert\{k\in[n]\mid i_k=r, j_k=1\}\vert.
\end{displaymath}
Then $\displaystyle(\sum_{r=0}^{p-1}\gamma_r)\leq n$ and
$\ty(\underline{i},\underline{j}):=(\gamma_0,\ldots,\gamma_{p-1})\in \NN^p$ 
is called the {\em weight} of $(\underline{i},\underline{j})$. 
For $\gamma\in \NN^p$, define
\begin{displaymath}
M[\gamma]:=\bigoplus_{\substack{\underline{i}\in \FF_{p}^n,\,
\underline{j}\in\{0,1\}^n\\\ty(\underline{i},\underline{j})=\gamma}} 
M_{\underline{i},\underline{j}}.
\end{displaymath}

The next lemma is motivated by the classical result that the center of 
$\FF[S_n]$ is generated by the elementary symmetric polynomials in 
Jucys--Murphy elements, see e.g. \cite{FHG} and \cite{Jucys}.

\begin{lemma}\label{lm:Rcenter}
The center of $\FF[R_n]$ is generated by the elementary symmetric 
polynomials in $\{X_{k}\mid k\in[n]\}$ and the elementary 
symmetric polynomials in $\{Y_{k}\mid k\in[n]\}$. 
\end{lemma}

Before we prove Lemma~\ref{lm:Rcenter}, we need to introduce some notation 
that we will use in the proof. For $\sigma\in R_n$, let $\mathcal{C}(\sigma)$ 
and $\mathcal{R}(\sigma)$ be the sets of indexes for all non-zero columns and
rows of $\sigma$, respectively. When $\mathcal{C}(\sigma)=\mathcal{R}(\sigma)$ and 
$|\mathcal{R}(\sigma)|=r$, there is a unique order preserving bijection 
$\mathbf{r}\to \mathcal{C}(\sigma)$ and $\sigma$ can be thought of as an 
element $\sigma'$ in $S_r$. We define the {\em cycle type} of 
$\sigma\in R_n$ as the cycle type of $\sigma'\in S_r$, see \cite{GM} for details.

\begin{example}
The element  
$$\sigma=\begin{bmatrix}
0 & 0 & 1\\
0 & 0 & 0\\
1 & 0 & 0
\end{bmatrix}\in R_3 $$ 
can be realized, by the above, as the transposition $(1,2)\in S_2$. 
So, the cycle type of this $\sigma$ is $(2)$, i.e. it has one cycle of length two.
\end{example}

\begin{proof}
Let us denote by $Z$ the subalgebra of $\Bbbk[R_n]$ spanned by all 
elementary symmetric polynomials in the $X_i$ and all 
elementary symmetric polynomials in the $Y_i$. It is a straightforward
exercise to check that $Z$ is contained in the center of $\Bbbk[R_n]$
(in fact, below we explicitly compute the elementary symmetric polynomials 
in the $X_i$ and the fact that these are central follow e.g. 
from \cite{Benadvances}).
Below we show the converse inclusion, i.e.  the center 
of $\Bbbk[R_n]$ is contained in $Z$.

For $ 0\leq r\leq n$ and $\lambda\vdash r$, let 
\begin{align*}
&\mathcal{M}_{\lambda}=\{\sigma\in R_{n}\mid \mathcal{C}(\sigma)=
\mathcal{R}(\sigma) \text{ and the cycle type of } \sigma=\lambda\}\\
&   c_{\lambda}= \sum_{\sigma\in \mathcal{M}_{\lambda}}E_{\mathcal{R}(\sigma)}\sigma.
\end{align*}
Note that $c_{\varnothing}=e_{ [n]}$ is the zero element of $R_n$.

Our first claim is that $\{c_{\lambda}\mid \lambda\vdash r, 0\leq r\leq n\}$ is a 
basis of the center of $\FF[R_n]$. This follows easily from the construction 
by combining the following three well-known facts:
\begin{itemize}
\item The center of $\FF[S_n]$ has the obvious basis indexed by the cycle types for $S_n$,
in which the basis element corresponding to a fixed cycle type  is just the
sum of all elements in $S_n$ which have this cycle  type.
\item For any $m$, the center the algebra of $m\times m$ matrices over $\FF[S_n]$ has the obvious
basis indexed by cycle types for $S_n$, in which the basis element corresponding to a 
fixed cycle type is just the identity matrix times the corresponding basis
element for  the center of $\FF[S_n]$.
\item Since $R_n$ is an inverse monoid, the monoid algebra $\FF[R_n]$ is isomorphic
to a direct sum of matrix algebras corresponding to the equivalence classes 
(with respect to Green's $\mathcal{D}$-relation) of the
maximal subgroups in $R_n$. The latter subgroups are of the form $S_k$, for $0\leq k\leq n$.
For each such $S_k$, the rows and columns in the corresponding matrix algebra 
are naturally indexes by all $k$-element subsets of $[n]$ and the
idempotents cutting out the $\Bbbk[S_k]$-parts are exactly the elements
$E_{\mathcal{R}(\sigma)}$, see e.g. \cite{Benadvances,Steinberg}.
\end{itemize}

Consider the element
$$d_{r}:=\displaystyle{\sum_{A\subseteq [n],\,|A|=r}} E_{A}.$$
Note that $X_1=Q-P=e_\varnothing-e_{\{1\}}=E_{\{1\}}$ and hence
$X_k=E_{\{k\}}$, for all $1\leq k\leq n$. In particular,
$d_1=X_1+X_2+\cdots+X_n$.
Further, for $1\leq i_1<i_2< \cdots <i_{r}\leq n$, we have 
$$X_{i_1}X_{i_2} \cdots X_{i_r}=E_{\{i_1\}}E_{\{i_2\}} \cdots E_{\{i_r\}}=
E_{\{i_1,i_2,\dots,i_r\}}$$ 
implying
$$d_{r}={\sum_{\substack{i_1,\ldots,i_r\in[n]\\ i_{1}< \cdots<i_r}} 
X_{i_1}X_{i_2} \cdots X_{i_r}},$$ 
which is an elementary symmetric polynomial in the elements $\{X_{k}\mid k\in[n]\}$. 

Next, by induction on $r$, one shows that the element
\begin{align*}
g_{r}:=&\displaystyle{\sum_{\substack{i_1,\ldots,i_r\in[n]\\ i_{1}< \cdots<i_r}}} e_{\{i_1\}}e_{\{i_2\}} \cdots e_{\{i_{r}\}}
\end{align*}
is a linear combination of $d_0,d_1,\dots,d_r$ and hence belongs to $Z$.

Further, note that 
\begin{align}\label{al:key}
E_{A}e_i=0 & \text{ if } i\in A.
\end{align}

For $\lambda\vdash r$, let $\tilde{\lambda}$ be the partition of $n$ obtained by adding $1^{(n-r)}$ to $\lambda$ at the end. Using \eqref{al:key}, we can write
\begin{align*}
	c_{\lambda}=d_{r} \bigg(\sum_{\substack{\tau\in S_n,\\ \text{ cycle type of }\tau=\tilde{\lambda}}} \tau\bigg)g_{n-r}.
	\end{align*}

By the classical results for symmetric groups, there exists a symmetric 
polynomial $f_{\lambda}$ in $n$ variables such that 
\begin{displaymath}
f_{\lambda}(\tilde{Y}_{1},\ldots,\tilde{Y}_n)=
\sum_{\substack{\tau\in S_n,\\ \text{ cycle type of }\tau=\tilde{\lambda}}} \tau.
\end{displaymath}
One again, using \eqref{al:key}, we conclude that $c_{\lambda}=d_rf_{\lambda}(Y_1,\ldots,Y_n)g_{n-r}\in Z$. 
\end{proof}

Note that, over $\mathbb{C}$, an analogue of Lemma~\ref{lm:Rcenter} 
is also true for the generalized rook monoid algebras:

\begin{proposition}
The center of $\CC[C_r\wr R_n]$ is generated by the symmetric polynomials in 
$\{X_i\mid i\in[n]\}$ and the symmetric polynomials in $\{Y_i\mid i\in[n]\}$.	
\end{proposition}

\begin{proof}
Indeed, from Theorem~\ref{thm:JM}, it is easy to see that 
all symmetric polynomials in the $X_i$ and all symmetric polynomials in
the $Y_i$ act as scalars on all simple modules and hence are central. 
At the same time, they separate the isomorphism classes of simple 
modules and hence generate the center.
\end{proof}

Over an arbitrary field, we only have the following:

\begin{proposition}\label{prop:gencenter}
The symmetric polynomials in $\{X_i\mid i\in[n]\}$ and the symmetric 
polynomials in $\{Y_i\mid i\in[n]\}$ belong to the center of $\FF[C_r\wr R_n]$.
\end{proposition}

\begin{proof}
To prove the result, it is enough to show that the elementary symmetric polynomials 
in $\{X_i\mid i\in[n]\}$ as well as the elementary symmetric polynomials in $\{Y_i\mid i\in[n]\}$ 
belong to the center of $\FF[C_r\wr R_n]$. 
Note that, both $P$ and $Q$ commute with each $X_i$ and $Y_i$. Furthermore, for $j\in[n-1]$, we have
\begin{itemize}
\item $ s_{i}X_is_i=X_{i+1}, \text{ and } s_jX_is_j=X_i \text{ for } j\notin \{i,i+1\},$
\item 	$ s_{i}Y_is_i=Y_{i+1}-E_{\{i,i+1\}}\displaystyle{\sum_{l=0}^{r-1}}\xi_{i}^{l}\xi_{i+1}^{-l}s_i, 
\text{ and } s_jY_is_j=Y_i \text{ for } j\notin \{i,i+1\}.$
\end{itemize}
From this it follows that the symmetric polynomials in question commute with the generators 
of $\FF[C_r\wr R_n]$ (see the end of Section~\ref{s2}) and hence  are central. 
\end{proof}

Now in order to avoid excessive notation, the remaining part of this subsection is for $r=1$ and for the arbitrary $r$ we have stated the corresponding results in Subsection \ref{sec:gen}.

Like in the case of symmetric groups, symmetric polynomials in $X_i$ and the symmetric polynomials in $Y_i$ being central (see Lemma~\ref{lm:Rcenter}) give the 
following statement.

\begin{lemma}\label{lm:Rblock}
Let $\mathcal{T}_{\leq n}:=\displaystyle{\bigcup^{n}_{i=0}} \mathcal{T}_{i}$. 
Then, any $M\in \FF[R_n]\Mod$ admits a decomposition 
$M\cong \displaystyle{\bigoplus_{\gamma\in\mathcal{T}_{\leq n}}}M[\gamma]$, 
as $\FF[R_n]$-modules.
\end{lemma}

We want to use Lemma \ref{lm:Rblock} to define, for $i\in \FF_p$,
the following functors:
\begin{align}\label{al:E}
&\res_i:\FF[R_n]\Mod\to \FF[R_{n-1}]\Mod,\quad \mathbb{A}:
\FF[R_n]\Mod\to \FF[R_{n-1}]\Mod, \\
&\ind_i:\FF[R_n]\Mod\to \FF[R_{n+1}]\Mod,\quad 
\mathbb{B}:\FF[R_n]\Mod \to \FF[R_{n+1}]\Mod.\nonumber
\end{align}
For $\gamma\in\displaystyle{\bigcup_{i=0}^n} \mathcal{T}_i$ and 
$M=M[\gamma]\in \FF[R_n]\Mod$, let
\begin{align}\label{al:A}
&{\res}_i(M[\gamma]):=\begin{cases}
(\Res^{\FF[R_n]}_{\FF[R_{n-1}]}M[\gamma])[\gamma-i], & \text{ if } \gamma_i\neq 0; \\
0, &\text{ otherwise};
\end{cases}\\ \nonumber
&\mathbb{A}(M[\gamma]):=\begin{cases}
(\Res^{\FF[R_n]}_{\FF[R_{n-1}]}(M[\gamma]))[\gamma],
& \text{ if } \gamma\in \mathcal{T}_{\leq n-1};\\
0, &\text{ otherwise};
\end{cases}\\ \nonumber
&
\ind_i(M[\gamma]):=
(\Ind^{\FF[R_{n+1}]}_{\FF[R_n]}M[\gamma])[\gamma+i]; 
\quad
\mathbb{B}(M[\gamma]):=(\Ind^{\FF[R_{n+1}]}_{\FF[R_n]}(M[\gamma]))[\gamma].
\end{align}
These definitions extend to any object in $\FF[R_n]\Mod$ by additivity 
using Lemma~\ref{lm:Rblock}. The functors $\res_i$ and $\ind_i$  
are called the {\em $i$-restriction} and 
{\em $i$-induction}, respectively. 
Using these definitions, we get the following decompositions 
\begin{displaymath}
\Res^{\FF[R_n]}_{\FF[R_{n-1}]}(V)\cong (\bigoplus_{i\in \FF_p}\res_i) 
\oplus {\mathbb{A}}\quad \text{and} \quad \Ind^{\FF[R_{n+1}]}_{\FF[R_{n}]}(V)
\cong (\bigoplus_{i\in \FF_p}\ind_i)\oplus{\mathbb{B}}. 
\end{displaymath}

Since $\res_{i},\mathbb{A},\ind_i$ and $\mathbb{B}$ are all 
exact functors, we get the induced $\CC$-linear maps 
$[\res_{i}],[\mathbb{A}],[\ind_i]$, and $[\mathbb{B}]$ on the  complexified
of Grothendieck groups.  
From Lemma \ref{lm:Requi}, for $r=1$, it follows that
$$\boldsymbol{\mathcal{Q}}_n:=\{[\mathcal{L}_{j}^{n}(D^{\lambda})]\mid 
\lambda\in\bigcup_{j=0}^{n}\Lambda^{p}(j)\}$$
is a basis of  $G_0(\FF[R_n])$. In order to describe 
$[\res_{i}],[\mathbb{A}],[\ind_i]$, and $[\mathbb{B}]$ in 
this basis, we use the following lemma, which says that via the functor in Lemma \ref{lm:Requi} the decomposition in Lemma \ref{lm:Sblock} goes to the decomposition in
Lemma \ref{lm:Rblock}.

\begin{lemma}\label{lm:RSblock}
For $0\leq j\leq n$, $\gamma\in\mathcal{T}_j$ and    
$V=V[\gamma]\in \FF[S_j]\Mod$, we have
\begin{align*}
\mathcal{L}_j^n(V[\gamma])=\mathcal{L}_j^n(V[\gamma])[\gamma].
\end{align*}
\end{lemma}

\begin{proof}
Clearly, $\mathcal{L}_j^n(V[\gamma])[\gamma]\subseteq \mathcal{L}_j^n(V[\gamma])$ and we prove the reverse inclusion. By definition, 
$\mathcal{L}_j^n(V[\gamma])=\FF\mathbb{L}_{j}^{n}\otimes_{\FF[S_j]}V[\gamma]$. For $\underline{r}=(r_1,r_2,\ldots,r_j)$ with $\ty(\underline{r})=\gamma$, 
let $v\in V_{\underline{r}}$. Let $Z=\{\beta_1< \cdots<\beta_j\}\subseteq [n]$.  Then using the decomposition \eqref{al:symdecom},  $h^{n}_Z\otimes v\in \mathcal{L}_j^n(V[\gamma])$ and a general element is a linear combination of these elements.  In the following, we show that $h^{n}_Z\otimes v\in \mathcal{L}^{n}_j(V[\gamma])[\gamma]$. 

Define $\underline{l}=(l_1,\ldots,l_n)$, where 
$l_{\beta_1}=l_{\beta_2}= \cdots =l_{\beta_j}=1$ 
and the remaining coordinates are equal to $0$. Also, let 
$\underline{m}=(m_1,\ldots,m_n)$, where $m_{\beta_1}=r_1,\ldots,m_{\beta_j}=r_j$ 
and the remaining coordinates are equal to $0$.  
Then $\ty((\underline{l},\underline{m}))=\gamma$.

 Let $k\in[n]$. Then
\begin{align*}
X_k  (h_{Z}^n\otimes v)=\begin{cases}
h_{Z}^n\otimes v ,&\text{ if } k\in Z;\\
0, & \text{ otherwise}. 
\end{cases}
\end{align*}
In particular, $(X_k-l_k)h_{Z}^n\otimes v=0$.   For $q_1,q_2\in [n]$, observe that
 \begin{align}\label{al:Eblock}
 	E_{\{q_1,q_2\}}h_Z^n=\begin{cases}
 		0, & \text{ if } q_1\notin Z \text{ or } q_2\notin Z;\\
 		h_Z^n, & \text{ otherwise}.
 			\end{cases}
 	\end{align}
 If $k\in Z$, then there exists $s\in[j]$ such that $\beta_s=k$. For $i\in[s-1]$, we observe that $(\beta_i,\beta_s)h^{n}_Z=h^{n}_Z(i,s)$. Now this together with \eqref{al:Eblock} imply that $Y_{k}h_{Z}^n=	h_{Z}^n\widetilde{Y}_s$.  If $k\notin Z$, then again from \eqref{al:Eblock}, we get $Y_kh_{Z}^n=0$.
 
 Since $v\in V_{\underline{r}}$, 
 by the definition of $V_{\underline{r}}$, we have $(\widetilde{Y}_q-r_q)^Nv=0$, for all $q\in[j]$ and  
 for $N\ge 0$. From the above, we get $(Y_{k}-m_k)^{N}  (h_{Z}^n\otimes v) =0$, for all $k\in[n]$.  Thus $h_{Z}^n\otimes v\in\mathcal{L}^{n}_i(V[\gamma])_{(\underline{l},\underline{m})}\subseteq \mathcal{L}^{n}_i(V[\gamma])[\gamma]$.
\end{proof}

The following corollary is  a consequence of  Theorem \ref{thm:res} and Lemma \ref{lm:RSblock}.
\begin{corollary}\label{coro:rel}
For $i\in\FF_p$ and $0\leq j\leq n$, we have
\begin{align*}
&\res_i\circ\mathcal{L}_j^n\cong \begin{cases}
\mathcal{L}_{j-1}^{n-1}\circ \widetilde{\res}_i, & \text{ if } j\neq0;\\
0, & \text{ otherwise};
\end{cases}
\quad\quad\quad \mathbb{A}\circ \mathcal{L}_j^n\cong\begin{cases}
\mathcal{L}_j^{n-1}, & \text{ if } j\neq n;\\
0, & \text{ otherwise};
\end{cases}\\
&\ind_i\circ\mathcal{L}_j^n \cong \mathcal{L}_{j+1}^{n+1}\circ \widetilde{\ind}_i \qquad \text{ and }
\qquad\mathbb{B}\circ \mathcal{L}_j^n\cong \mathcal{L}_{j}^{n+1}.
\end{align*}
\end{corollary}

As an application of Lemma \ref{lm:Requi} and Corollary \ref{coro:rel}, we obtain the following corollary.
\begin{corollary}\label{coro:commute}
For $i\in \FF_p$, we have
\begin{align*}
&\res_i\circ\mathbb{A}\cong\mathbb{A}\circ\res_i,\,
\ind_i\circ\mathbb{A}\cong\mathbb{A}\circ\ind_i, \, \res_i\circ\mathbb{B}\cong\mathbb{B}\circ\res_i, \\ & \ind_i\circ\mathbb{B}\cong\mathbb{B}\circ\ind_i, \text{ and also } \mathbb{A}\circ\mathbb{B}\cong \mathrm{Id}.
\end{align*}
	\end{corollary}

Let $\lambda\in\displaystyle{\bigcup_{j=0}^{n}}\Lambda^p(j)$. For $i\in \FF_p$, below in \eqref{eq:E} and  \eqref{eq:F}, the first equality is due to Corollary \ref{coro:rel} and the second equality is due to Theorem \ref{thm:action}
\begin{align}
&[\res_i]([\mathcal{L}_{j}^n(D^\lambda)])=[\mathcal{L}_j^n]([\widetilde{\res}_i][D^\lambda])=
\sum_{\mu\in\Lambda^{p}(n-1)}\alpha_{\lambda\mu}
[\mathcal{L}_{j-1}^{n-1}(D^\mu)],\label{eq:E}\\
&	[\ind_i]([\mathcal{L}_{j}^{n}(D^\lambda)])=[\mathcal{L}_j^n]([\widetilde{\ind}_i][D^\lambda])=
\sum_{\nu\in\Lambda^{p}(n+1)}\beta_{\lambda\nu}
[\mathcal{L}_{j+1}^{n+1}(D^\nu)], \label{eq:F}
\end{align}
where $\alpha_{\lambda\mu}$ and $\beta_{\lambda\nu}$ are as given in Theorem \ref{thm:action}.  Once again from Corollary \ref{coro:rel}, we obtain
\begin{align}
&[\mathbb{A}]([\mathcal{L}_{j}^n(D^\lambda)])=\begin{cases} 
[\mathcal{L}_{j}^{n-1}(D^\lambda)], & \text{ if  } j\neq n;\\
0,  &\text{  otherwise};
\end{cases}\label{eq:A}\\
&[\mathbb{B}]([\mathcal{L}_{j}^{n}(D^\lambda)])=
[\mathcal{L}_{j}^{n+1}(D^\lambda)].\label{eq:B}
\end{align}

Define $$\mathcal{G}_{\CC}:=\displaystyle{\bigoplus_{n\in\NN}}G_0(\FF[R_n]).$$ 
Then we can view  $[\res_{i}],[\mathbb{A}],[\ind_i]$ and $[\mathbb{B}]$ 
as endomorphisms on $\mathcal{G}_{\CC}$.

\begin{theorem}\label{thm:main}
\hspace{1mm}

\begin{enumerate}[$($a$)$]
\item\label{thm:main.1} For $i\in \FF_p$, the endomorphisms $[\res_i]$ and
$[\ind_i]$ on $\mathcal{G}_{\CC}$ satisfy the defining relations 
of the Chevalley generators of $\widehat{\mathfrak{sl}}_{p}(\CC)$.
\item\label{thm:main.2}	For $i\in \FF_p$, we have the relations 
\begin{align*}
&[\res_i][\mathbb{A}]=[\mathbb{A}][\res_i], \,[\res_i]
[\mathbb{B}]=[\mathbb{B}][\res_{i}], \,
[\ind_i][\mathbb{A}]=[\ind_{i}][\mathbb{A}],\,[\ind_i]
[\mathbb{B}]=[\ind_{i}][\mathbb{B}],\\  &\text{ and also }
[\mathbb{A}][\mathbb{B}]=\mathrm{Id}_{\mathcal{G}_{\CC}}.
\end{align*}

\item \label{thm:main.3} The vector space $\mathcal{G}_{\CC}$ is a 
module over $U(\widehat{\mathfrak{sl}}_p(\CC))\otimes_{\CC} \CC[\mathcal{B}]$, moreover, 
\begin{align*}
\mathcal{G}_{\CC}	\cong V(\Lambda_{0})\otimes_{\CC} V_{\NN},\quad \text{ as } 
(U(\widehat{\mathfrak{sl}}_p(\CC))\otimes_{\CC} \CC[\mathcal{B}])\text{-}modules.
\end{align*}\end{enumerate}
\end{theorem}

\begin{proof}
Claim~\eqref{thm:main.1} follows from Theorem~\ref{thm:action}, 
Theorem~\ref{thm:Smod}\eqref{thm:Smod.1}, and Formulae~\eqref{eq:E} and \eqref{eq:F}. Claim~\eqref{thm:main.2} follows from Corollary \ref{coro:commute}. 

Let us now prove Claim~\eqref{thm:main.3}.
Note that  $U(\widehat{\mathfrak{sl}}_p(\CC))\otimes\CC[\mathcal{B}]$ is generated by  
$e_{i}\otimes 1$, $f_i\otimes 1$,  for all $i\in \FF_p$, and by 
$1\otimes a$, $1\otimes b$. Let $e_{i}\otimes 1 $ and $f_i\otimes 1$ 
act on $\mathcal{G}_{\CC}$ by $[\res_i]$ and $[\ind_i]$, 
respectively, for all $i\in \FF_p$. Likewise, let $1\otimes a$ and 
$1\otimes b$ act on $\mathcal{G}_{\CC}$ by $[\mathbb{A}]$ and 
$[\mathbb{B}]$, respectively. Then using claims~\eqref{thm:main.1}
and \eqref{thm:main.2}, we see  that $\mathcal{G}_{\CC}$ is a module over 
$U(\widehat{\mathfrak{sl}}_{p}(\CC))\otimes_{\CC} \CC[\mathcal{B}]$.

Under the isomorphism in Theorem~\ref{thm:Smod}\eqref{thm:Smod.2}, 
we may consider  
$$\{[D^{\lambda}]\mid \lambda\in {\displaystyle \bigcup_{j\in\NN}}\Lambda^p(j)\}$$
as a basis of $V(\Lambda_{0})$. Define the map,
\begin{align}\label{iso:groth}
&	\Phi:\mathcal{G}_{\CC}\to V{(\Lambda_{0})}\otimes V_{\NN}, \text{ by }\\
&	\Phi([\mathcal{L}^{n}_j(D^{\lambda})])=[D^{\lambda}]\otimes (n-j).\nonumber
\end{align}
It follows from the above discussion and the constructions that this map 
is an isomorphism of $U(\widehat{\mathfrak{sl}}_{p}(\CC))\otimes_{\CC}\CC[\mathcal{B}]$-modules.
\end{proof}

\subsection{Bialgebra structure on \texorpdfstring{$\mathcal{G}_{\CC}$}{}}

\subsubsection{Preliminaries}

It is well known that $\tilde{\mathcal{G}}_{\CC}$ has the natural structure of  
a Hopf algebra, see \cite[Chapter I]{Mac}. In this section, we prove that $\mathcal{G}_{\CC}$ 
has the natural structure of a bialgebra.

For $j,k\in\mathbb{N}$, denote  $\FF[S_{(j,k)}]:=\FF[S_j]\otimes_\Bbbk \FF[S_{k}]$.
Further,  for $a,b\in\mathbb{N}$, denote
$\mathcal{L}_{(j,k)}^{(a,b)}:=(\FF\mathbb{L}_{j}^{a}\otimes_\Bbbk \FF\mathbb{L}_k^{b})
\displaystyle{\otimes _{\FF[S_{(j,k)}]}}{}_-$.

\begin{lemma}\label{lm:Res}
For $n,n_1,n_2\in\NN$ with $n=n_1+n_2$, the following  diagram 
\begin{displaymath}
\xymatrixcolsep{9pc}
\xymatrix{	
\displaystyle{\bigoplus_{i=0}^{n}} \FF[S_i]\Mod\ar[d]^{\mathcal{F}}
\ar[r]<0pt>^{\displaystyle{\bigoplus_{i=0}^{n}} \mathcal{L}^{n}_{i}}
& \FF[R_n]\Mod\ar[d]^{\Res^{\FF[R_n]}_{\FF[R_{(n_1,n_2)}]}({}_-)}\\
\displaystyle{\bigoplus_{j=0}^{n_1}} 	\displaystyle{\bigoplus_{k=0}^{n_2}} \FF[S_{(j,k)}]\Mod
\ar[r]<0pt>^{\displaystyle{{\bigoplus_{j=0}^{n_1}} 	\displaystyle{\bigoplus_{k=0}^{n_2}}}
\mathcal{L}_{(j,k)}^{(n_1,n_2)}}
&\FF[R_{(n_1,n_2)}]\Mod,
}
\end{displaymath}
where  the functor $\mathcal{F}$ is given by
\begin{align*}
\mathcal{F}|_{\FF[S_i]\Mod}=
\bigoplus_{\substack{j\in\{0,\dots,n_1\}\\k\in\{0,\dots,n_2\}\\j+k=i}}
\Res^{\FF[S_i]}_{\FF[S_{(j,k)}]}({}_-).
\end{align*}
commutes up to isomorphism of functors.
\end{lemma}

\begin{proof}
Recall that, as a right $\FF[C_{r}\wr S_i]$-module, $\FF\mathbb{L}_i^n$ has a basis 
consisting of matrices of the form $h_Z^n$, where $Z\in\mathcal{S}_i$. 

For $Z$ as above, let
$$Z'=Z\cap\{1,2,\ldots,n_1\} \text{ and } Z''=Z\cap\{n_1+1,n_1+2,\ldots,n_1+n_2=n\}.$$ 
Then $Z=Z'\sqcup Z''$. If $|Z'|=j$ and $|Z''|=k$, then we have 
$0\leq j\leq n_1$ and $0\leq k\leq n_2$ such that  $j+k=i$.

Then the map

\resizebox{\textwidth}{!}{
$
\begin{array}{rcl}
\Res^{\FF[R_n]}_{\FF[R_{(n_1,n_2)}]}(\mathcal{L}_i^n(V))=
\Bbbk\mathbb{L}_i^n\otimes_{\FF[S_i]} V & \to&
\displaystyle\bigoplus_{\substack{j\in\{0,\dots,n_1\}\\k\in\{0,\dots,n_2\}\\j+k=i}}
(\Bbbk\mathbb{L}_j^{n_1}\otimes_\Bbbk \Bbbk\mathbb{L}_k^{n_2})\otimes_{\FF[S_{(j,k)}]}
\Res^{\FF[S_i]}_{\FF[S_{(j,k)}]}(V) \\
h_{Z}^{n}\otimes v&\mapsto &(h_{Z'}^{n_1}\otimes h_{Z''}^{n_2})\otimes v
\end{array}
$
}
is an isomorphism of $\FF[R_{(n_1,n_2)}]$-modules which is  functorial in $V$. 
The claim follows.
\end{proof}

The following statement follows from  Lemma~\ref{lm:Res} using Frobenius reciprocity.

\begin{corollary}\label{cor:Ind}
For $n,n_1,n_2\in\NN$ with $n=n_1+n_2$, the following diagram
\begin{displaymath}
\xymatrixcolsep{9pc}
\xymatrix{	
\displaystyle{\bigoplus_{ i=0}^{n}} \FF[S_i]\Mod 
\ar[r]<0pt>^{\displaystyle{\bigoplus_{i=0}^{n}} \mathcal{L}^{n}_{i}}
& \FF[R_n]\Mod\\
\displaystyle{\bigoplus_{j=0}^{n_1}} 	
\displaystyle{\bigoplus_{k=0}^{n_2}} \FF[S_{(j,k)}]\Mod 
\ar[r]<0pt>^{  \displaystyle{\bigoplus_{j=0}^{n_1}} 	
\displaystyle{\bigoplus_{k=0}^{n_2}}   \mathcal{L}_{(j,k)}^{(n_1,n_2)} }\ar[u]_{\mathcal{F}'}
&\FF[R_{(n_1,n_2)}]\Mod\ar[u]_{\Ind^{\FF[R_n]}_{\FF[R_{(n_1,n_2)}]}({}_-)},
}
\end{displaymath}
where  the functor $\mathcal{F}'$ is given by
\begin{align*}
\mathcal{F}'|_{\FF[S_i]\Mod}=
\displaystyle{\bigoplus_{\substack{j\in\{0,\dots,n_1\}\\k\in\{0,\dots,n_2\}\\j+k=i}}}
\Ind^{\FF[S_i]}_{\FF[S_{(j,k)}]}({}_-)
\end{align*}
commutes up to isomorphism of functors.
\end{corollary}

An immediate consequence of Corollary~\ref{cor:Ind} is the following statement.

\begin{corollary}
For $n,n_1,n_2\in\NN$ with $n=n_1+n_2$, the functor 
$\Ind_{\FF[R_{(n_1,n_2)}]}^{\FF[R_n]}({}_-)$ is exact.
\end{corollary}

For $n_1,n_2,\ldots,n_s\in\NN$, let 
$R_{(n_1,n_2,\ldots,n_s)}:=R_{n_1}\times R_{n_2}\times \cdots\times R_{n_s}$. 
Then, as  usual, we have the following decomposition involving the corresponding monoid algebras:
\begin{displaymath} 
\FF[R_{(n_1,n_2,\ldots,n_s)}]=\FF[R_{n_1}]\otimes_\Bbbk \FF[R_{n_2}]\otimes_\Bbbk 
 \cdots\otimes_\Bbbk \FF[R_{n_s}].
\end{displaymath}
Now we are ready to discuss the bialgebra structure on $\mathcal{G}_{\CC}$.

\subsubsection{Multiplication}
Since we have to deal with modules over $\FF[R_n]$ for all $n\in \NN$ in the same course of a proof or a statement, for the sake of clarity we decorate a module over $\FF[R_n]$ by putting superscript $n$ on it. This notational convention applies for modules over symmetric group algebras as well. For $V^{n}\in \FF[R_n]\Mod$ and $W^{m}\in \FF [R_m]\Mod$, 
we have that $V^n\otimes_\Bbbk W^m\in \FF[R_{(n,m)}]\Mod$. Define:
\begin{align}\label{al:grothmulti}
[V^{n}] [W^{m}]=[\Ind^{\FF[R_{n+m}]}_{\FF[R_{(n,m)}]}(V^{n}\otimes_{\FF} W^{m})].
\end{align}
Since the functor $\Ind^{\FF[R_{n+m}]}_{\FF[R_{(n,m)}]}(-)$ is exact, 
$\eqref{al:grothmulti}$ gives rise to a well-defined multiplication on $\mathcal{G}_{\CC}$.  
Associativity of tensor products and also of the induction functor imply  the associativity of $\eqref{al:grothmulti}$. 
Let $\FF^{0}$ denote the trivial $\FF[R_{0}]$-module.
Then $[\FF ^0]\in G_0( \FF[R_0])$ is the unit with respect to this multiplication.
Thus $\mathcal{G}_{\CC}$ becomes a unital algebra with respect to the 
multiplication given by \eqref{al:grothmulti}.

\subsubsection{Comultiplication}	

Define 
\begin{align}\label{al:grothcomulti}
\Delta([V^{n}])=\sum_{\substack{n_1,n_2\in \mathbb{N}\\
n_1+n_2=n}}[\Res^{\FF[R_n]}_{\FF[R_{(n_1,n_2)}]}V^n].
\end{align}
Using the identification
\begin{align*}
	\bigoplus_{n_1,n_2\in \NN}G_0(\FF[R_{(n_1,n_2)}])\cong 
	\mathcal{G}_{\CC}\otimes_{\CC} \mathcal{G}_{\CC},
\end{align*}
$\Delta([V^{n}])\in\mathcal{G}_\CC\otimes_{\CC}\mathcal{G}_\CC$. Then, both sides of the coassociativity condition for \eqref{al:grothcomulti} 
reduce, essentially, to computation of every 
$\Res^{\FF[R_n]}_{\FF[R_{(n_1,n_2,n_3)}]}(V^n)$, for $n=n_1+n_2+n_3$. Consider the map 
$\epsilon:\mathcal{G}_{\CC }\to \FF$ which sends the basis element 
$[\FF^{0}]\in G_0(\FF[R_{0}])$ to $1\in \FF$ and is zero on all other basis elements. 
It is straightforward that the map $\epsilon$ is a counit of $\mathcal{G}_{\CC}$, and 
so $\mathcal{G}_{\CC}$  becomes a coalgebra with respect to $\Delta$ and $\epsilon$.

\subsubsection{Compatibility}	

The vector space $V_{\NN}$ is isomorphic to the monoid algebra $\CC[\NN]$ of the monoid 
$(\NN,+)$ of natural numbers. Therefore $V_{\NN}$ inherits from $\CC[\NN]$ the structure 
of a bialgebra, where the multiplication is given by 
the monoid operation (addition) and the value of the comultiplication on $i\in \NN$ 
is  $\displaystyle{\sum_{\substack{i_1,i_2\in \NN\\ i_1+i_2=i}}}{i_1\otimes i_2}$. 
It is well-known that $V(\Lambda_{0})$ is a Hopf algebra, where the multiplication 
and comultiplication are given by replacing $R_n$ by $S_n$ in $\eqref{al:grothmulti}$ 
and in $\eqref{al:grothcomulti}$, respectively. As a consequence, 
we obtain that $V(\Lambda_{0})\otimes_\CC V_{\NN}$ is, naturally, a bialgebra.

Next we prove that the respective multiplication and comultiplication maps are 
preserved under the isomorphism \eqref{iso:groth}. In particular, this  implies that 
$\Delta$ is compatible with multiplication and thereby $\mathcal{G}_{\CC}$  possess the 
structure of a bialgebra.

\begin{theorem}\label{thm:bi}
The isomorphism \eqref{iso:groth} $\Phi$ preserves multiplication and 
comultiplication. In particular, $\mathcal{G}_{\CC}$  is a bialgebra.
\end{theorem}

\begin{proof}
For the comultiplication maps $\Delta_{V(\Lambda_{0})}$ and $\Delta_{V_{\NN}}$ 
of $V{(\Lambda_{0})}$ and $V_{\NN}$, respectively, the comultiplication on 
$V(\Lambda_{0})\otimes_{\CC} V_{\NN}$ is given by
\begin{displaymath}
\Delta_{ V(\Lambda_{0})\otimes_{\CC} V_{\NN}}:=(\id_{V(\Lambda_{0})}
\otimes\tau\otimes\id_{V_{\NN}})\circ(\Delta_{V(\Lambda_{0})}\otimes \Delta_{V_{\NN}}),
\end{displaymath}
where $\tau:V{(\Lambda_{0})}\otimes_\CC V_{\NN}\to V_{\NN}\otimes_\CC V{(\Lambda_{0})}$	
is  the swap of the tensor factors.

For $i\leq n$, let $M^{i}\in \FF[S_i]\Mod$. We want to show that  
\begin{align}\label{al:comulti}
((\Phi\otimes\Phi)\circ\Delta)([\mathcal{L}_{i}^n(M^{i})])=
(\Delta_{ V(\Lambda_{0})\otimes_{\CC} V_{\NN}}\circ\Phi)([\mathcal{L}_{i}^n(M^{i})]).
\end{align}

Under the identification
\begin{align*}
	\bigoplus_{n_1,n_2\in \NN}G_0(\FF[S_{(n_1,n_2)}])\cong 
	\tilde{\mathcal{G}}_{\CC}\otimes_{\CC} \tilde{\mathcal{G}}_{\CC},
\end{align*}
we fix a decomposition of $[\Res^{\FF[S_i]}_{\FF[S_{(j,k)}]}(M^i)]$ of the form
$\displaystyle \sum_{a,b\in \NN} [V_a^j]\otimes[V_b^k]$ (where all but finitely many
summands are zero). Now, using Lemma~\ref{lm:Res},
the left hand side of \eqref{al:comulti} can be computed as follows:
\begin{align*}	
&((\Phi\otimes\Phi)\circ\Delta)([\mathcal{L}_{i}^n(M^{i})])\\&=
\Phi\otimes\Phi\bigg(\sum_{\substack{n_1,n_2\in\NN\\ n_1+n_2=n}}
[\Res^{\FF[R_n]}_{\FF[R_{(n_1,n_2)}]}(\mathcal{L}_i^{n}(M^i))]\bigg)\\
&= \Phi\otimes\Phi \bigg(\sum_{\substack{n_1,n_2\in\NN\\ n_1+n_2=n}}
\bigg[\bigoplus_{\substack{j\in\{0,\ldots,n_1\}\\k\in\{0,\ldots , n_2\}\\ j+k=i}}
\mathcal{L}_{(j,k)}^{(n_1,n_2)}\circ \mathcal{F}(M^i)\bigg]\bigg)\\
&=\Phi\otimes\Phi \bigg(\sum_{\substack{n_1,n_2\in\NN\\ n_1+n_2=n}}
\bigg[\bigoplus_{\substack{j\in\{0,\ldots,n_1\}\\k\in\{0,\ldots,n_2\}\\ j+k=i}}
\mathcal{L}_{(j,k)}^{(n_1,n_2)}\Res^{\FF[S_i]}_{\FF[S_{(j,k)}]}(M^i)\bigg]\bigg)\\
&=\Phi\otimes\Phi \bigg(\sum_{\substack{n_1,n_2\in\NN\\ n_1+n_2=n}}
\sum_{\substack{j\in\{0,\ldots,n_1\}\\k\in\{0,\ldots,n_2\}\\ j+k=i}}
\bigg[\mathcal{L}_{(j,k)}^{(n_1,n_2)}\Res^{\FF[S_i]}_{\FF[S_{(j,k)}]}(M^i)\bigg]\bigg)\\
&=\Phi\otimes\Phi \bigg(\sum_{\substack{n_1,n_2\in\NN\\ n_1+n_2=n}}
\sum_{\substack{j\in\{0,\ldots,n_1\}\\k\in\{0,\ldots,n_2\}\\ j+k=i}}
\bigg[\mathcal{L}_{(j,k)}^{(n_1,n_2)}\bigg] \bigg[\Res^{\FF[S_i]}_{\FF[S_{(j,k)}]}(M^i)\bigg]\bigg)\\
&=\Phi\otimes\Phi \bigg(\sum_{\substack{n_1,n_2\in\NN\\ n_1+n_2=n}}
\sum_{\substack{j\in\{0,\ldots,n_1\}\\k\in\{0,\ldots,n_2\}\\ j+k=i}}
\bigg[\mathcal{L}_{(j,k)}^{(n_1,n_2)}\bigg] \sum_{a,b\in \NN} [V_a^j]\otimes[V_b^k]\bigg)\\
&=\Phi\otimes\Phi \bigg(\sum_{\substack{n_1,n_2\in\NN\\ n_1+n_2=n}}\sum_{\substack{j\in\{0,\ldots,n_1\}\\k\in\{0,\ldots,n_2\}\\ j+k=i}}\sum_{a,b\in \NN}[\mathcal{L}_j^{n_1}(V_a^j)]\otimes[\mathcal{L}_k^{n_2}(V_b^k)]\bigg)\\
&=\sum_{\substack{n_1,n_2\in\NN\\ n_1+n_2=n}}\sum_{\substack{j\in\{0,\ldots,n_1\}\\k\in\{0,\ldots,n_2\}\\ j+k=i}}\sum_{a,b\in \NN}\Phi[(\mathcal{L}_j^{n_1}(V_a^j))]\otimes\Phi[(\mathcal{L}_k^{n_2}(V_b^k))]\\
&=\sum_{\substack{n_1,n_2\in\NN\\ n_1+n_2=n}}\sum_{\substack{j\in\{0,\ldots,n_1\}\\k\in\{0,\ldots,n_2\}\\ j+k=i}}\sum_{a,b\in \NN}\bigg([V_a^j]\otimes (n_1-j)\bigg)\otimes\bigg([V_b^k]\otimes(n_2-k)\bigg).
\end{align*}

On the other hand, the right hand side of \eqref{al:comulti}
can be computed as follows:
\begin{align*}
&(\Delta_{ V(\Lambda_{0})\otimes_{\CC} V_{\NN}}\circ\Phi)([\mathcal{L}_{i}^n(M^{i})])\\
&=\Delta_{ V(\Lambda_{0})\otimes_{\CC} V_{\NN}}([M^{i}]\otimes (n-i))\\
&=(\id_{V(\Lambda_{0})}\otimes\tau\otimes\id_{V_{\NN}})
\bigg(\Delta_{V(\Lambda_{0})}[M^{i}]\otimes \Delta_{V_{\NN}}(n-i)\bigg)\\
&=(\id_{V(\Lambda_{0})}\otimes\tau\otimes\id_{V_{\NN}})\left(
\bigg(\sum_{\substack{p,q\in\NN\\ p+q=i}}[\Res^{\FF[S_i]}_{\FF[S_{(p,q)}]}(M^i)]\bigg)
\otimes \bigg(\sum_{\substack{r,s\in\NN\\ r+s=n-i}}r\otimes s\bigg)\right)\\
&=(\id_{V(\Lambda_{0})}\otimes\tau\otimes\id_{V_{\NN}})\left(
\bigg(\sum_{\substack{p,q\in\NN\\ p+q=i}}\sum_{c,d\in\NN}[V_c^{p}]
\otimes [V_d^{q}]\bigg)\otimes \bigg(\sum_{\substack{r,s\in\NN\\ r+s=n-i}}r\otimes s\bigg)\right)\\
&=\sum_{c,d\in\NN}\sum_{\substack{p,q\in\NN\\ p+q=i}}
\bigg(\sum_{\substack{r,s\in\NN\\ r+s=n-i}}([V_c^p]\otimes r)\otimes ([V_d^q]\otimes s)\bigg).
\end{align*}
This implies \eqref{al:comulti} and we are done.
\end{proof}

\subsection{The case of generalized rook monoids}\label{sec:gen}

Suppose $p$ does not divide $r$. For $r\in [n]$, we have $X_i^{r}=X_{i}$ 
and hence the eigenvalues of each $X_i$, considered as an operator on a 
$\FF[C_{r}\wr R_n]$-finite-dimensional module, are  either 
$r$-th roots of unity or $0$. Similarly, the 
eigenvalues of Jucys--Murphy elements $Y_i$ as an operator 
on a finite dimensional module over $\FF[C_r\wr R_n]$ lie in $\FF_p$.  Using this and Proposition \ref{prop:gencenter}, we get a decomposition of every object in $\FF[C_r\wr R_n]\Mod$ as in Lemma.  
This allows us to define the $i$-induction and $i$-restriction 
functors as well as the functors corresponding to the two generators 
of the bicyclic monoid $\mathcal{B}$. Using the 
results for the generalized symmetric groups $C_r\wr S_n$ from
\cite{ShuJ} (see also \cite{Wang} for the more general setting), one shows that the direct sum, over all $n$, 
of  $G_0(\FF[C_r\wr R_n])$ is a
$U(\widehat{\mathfrak{sl}}_{p}(\CC))^{\otimes r}\otimes \CC[\mathcal{B}]$-module
isomorphic to 
$V(\Lambda_{0})^{\otimes r}\otimes V_{\NN}$.

\begin{remark}
The results of Section~\ref{sec:gro} have the obvious analogues in characteristic zero (with the same proofs), 
where the field $\FF_p$ is replaced by the ring $\ZZ$ of integers and, consequently,
the Lie algebra $\widehat{\mathfrak{sl}}_{p}(\CC)$ is replaced by $\widehat{\mathfrak{sl}}_{\infty}(\CC)$.  Also, the basic representation $V(\Lambda_0)$ now becomes the Fock space representation of $\widehat{\mathfrak{sl}}_{\infty}(\CC)$.
\end{remark}

	{\footnotesize
\bibliographystyle{alphaurl}

}

\end{document}